\documentclass[12pt,a4paper]{amsart}

\usepackage{amssymb}
\usepackage{amscd}
\usepackage{hyperref}

\def\frak{\mathfrak}
\def\Bbb{\mathbb}
\def\Cal{\mathcal}

\let\phi\varphi

\newcommand{\x}{\times}
\renewcommand{\o}{\circ}

\newcommand{\al}{\alpha}
\newcommand{\be}{\beta}
\newcommand{\ga}{\gamma}

\newcommand{\ep}{\epsilon}
\newcommand{\ka}{\kappa}
\newcommand{\la}{\lambda}

\newcommand{\ph}{\phi}
\newcommand{\ps}{\psi}
\renewcommand{\th}{\theta}
\newcommand{\si}{\sigma}
\newcommand{\ze}{\zeta}
\newcommand{\Ga}{\Gamma}
\newcommand{\La}{\Lambda}
\newcommand{\Ph}{\Phi}
\newcommand{\Ps}{\Psi}
\newcommand{\Om}{\Omega}

\newcommand{\Ups}{\Upsilon}
\def\Rho{\mbox{\textsf{P}}}

\newcommand{\barm}{\overline{M}}

\newcommand{\id}{\operatorname{id}}
\newcommand{\Fl}{\operatorname{Fl}}

\newcommand{\End}{\operatorname{End}}

\newcommand{\rpl}                         
{\mbox{$
\begin{picture}(12.7,8)(-.5,-1)
\put(0,0.2){$+$}
\put(4.2,2.8){\oval(8,8)[r]}
\end{picture}$}}

\newcounter{theorem}
\newtheorem{thm}[theorem]{Theorem}
\newtheorem*{thm*}{Theorem \thesubsection}
\newtheorem{lemma}[theorem]{Lemma}
\newtheorem{prop}[theorem]{Proposition}

\newtheorem*{lemma*}{Lemma \thesubsection}
\newtheorem*{prop*}{Proposition \thesubsection}
\newtheorem*{cor*}{Corollary \thesubsection}

\theoremstyle{definition}
\newtheorem{definition}[theorem]{Definition}
\newtheorem*{definition*}{Definition \thesubsection}

\newtheorem*{example*}{Example \thesubsection}
\theoremstyle{remark}

\newtheorem*{remark*}{Remark \thesubsection}

\def\sideremark#1{\ifvmode\leavevmode\fi\vadjust{\vbox to0pt{\vss
 \hbox to 0pt{\hskip\hsize\hskip1em
 \vbox{\hsize3cm\tiny\raggedright\pretolerance10000
  \noindent #1\hfill}\hss}\vbox to8pt{\vfil}\vss}}}%
                        
\begin{document}

\title{Projective compactness\\ and conformal boundaries}

\author{Andreas \v Cap and A.\ Rod Gover}

\address{A.\v C.: Faculty of Mathematics\\
University of Vienna\\
Oskar--Morgenstern--Platz 1\\
1090 Wien\\
Austria\\
A.R.G.:Department of Mathematics\\
  The University of Auckland\\
  Private Bag 92019\\
  Auckland 1142\\
  New Zealand;\\
Mathematical Sciences Institute\\
Australian National University \\ ACT 0200, Australia} 
\email{Andreas.Cap@univie.ac.at}
\email{r.gover@auckland.ac.nz}

\begin{abstract}
Let $\overline{M}$ be a smooth manifold with boundary $\partial M$ and
interior $M$.  Consider an affine connection $\nabla$ on $M$ for which
the boundary is at infinity. Then $\nabla$ is projectively compact of
order $\alpha$ if the projective structure defined by $\nabla$
smoothly extends to all of $\overline{M}$ in a specific way that
depends on no particular choice of boundary defining function. Via the
Levi--Civita connection, this concept applies to pseudo--Riemannian
metrics on $M$. We study the relation between interior geometry and
the possibilities for compactification, and then develop the tools
that describe the induced geometry on the boundary.

We prove that a pseudo--Riemannian metric on $M$ which
is projectively compact of order two admits a certain asymptotic
form. This form was known to be sufficient for projective
compactness, so the result establishes that it provides an equivalent
characterization.

From a projectively compact connection on $M$, one obtains a projective
structure on $\overline{M}$, which induces a conformal class of
(possibly degenerate) bundle metrics on the tangent bundle to the
hypersurface $\partial M$. Using the asymptotic form, we prove that in
the case of metrics, which are projectively compact of order two, this
boundary structure is always non--degenerate. We also prove that in
this case the metric is necessarily asymptotically Einstein, in a
natural sense.

Finally, a non--degenerate boundary geometry gives rise to a
(conformal) standard tractor bundle endowed with a canonical linear
connection, and we explicitly describe these in terms of the projective
data of the interior geometry.
\end{abstract}


\maketitle

\pagestyle{myheadings} \markboth{\v Cap, Gover}{Projective
  compactness} 

\thanks{Both authors gratefully acknowledge support from the Royal
  Society of New Zealand via Marsden Grant 13-UOA-018; A\v C
  gratefully acknowledges support by projects P23244-N13 and
  P27072-N25 of the Austrian Science Fund (FWF) and also the
  hospitality of the University of Auckland. }

\section{Introduction}\label{1}
Consider a smooth manifold $\barm$ with boundary $\partial M$ and
interior $M$. The study of geometric structures on $\partial M$
induced by complete Riemannian (or pseudo--Riemannian) metrics on $M$,
and of the relation between asymptotic data on $M$ and data on
$\partial M$, has a long history that includes interesting applications in
mathematics and physics, e.g. \cite{CY,GrL,HS}. A model case for this
situation is provided by conformally compact metrics on $M$,
i.e.~complete metrics for which an appropriate conformal rescaling
extends to the boundary. Such a metric gives rise to a well defined
conformal class of metrics on $\partial M$, which then is referred to
as the conformal infinity of the interior metric. Originating and
flourishing in general relativity (see
e.g.~\cite{Chrusciel,Fr,Friedrich,P-orig,Penrose125}), this concept
has also found important applications in geometric scattering theory
(\cite{GrZ,Ma-hodge,Melrose}) and the conjectural AdS--CFT
correspondence in physics (\cite{AdSCFTreview,deHaro}). If one in
addition requires the conformally compact metric on $M$ to be negative
Einstein, one arrives at the notion of a Poincar\'e--Einstein
metric. Realizing a given conformal class on a manifold formally as
the conformal infinity of a Poincar\'e--Einstein metric is closely
related to the Fefferman--Graham conformal ambient metric construction
\cite{FG1,FG2}, and so provides a central tool for generating conformal
invariants.

In fact the ambient metric construction is most
directly related to projective differential geometry.  As part of a
discussion of this point in \cite[Chapter 4]{FG2} Fefferman and Graham
present a certain asymptotic form for pseudo--Riemannian metrics,
which they call projectively compact, and they observe that
appropriate projective modifications of the Levi--Civita connections
of these admit smooth extensions to the boundary.  They did not go
further into the relations to projective differential geometry,
however.  On the other hand, in the classical and visionary articles
\cite{SH1,SH2}, Schouten and Haantjes develop a construction
essentially equivalent to the ambient metric, but based on projective
differential geometry.

In another case of the implicit use of projective geometry, a replacement of
conformal compactification by projective compactification is the
geometric move underlying the significant advances in the microlocal
analysis of asymptotically hyperbolic and de Sitter spaces
recently developed by Vasy in \cite{Vforms,Vasy}.

Projective compactification is potentially extremely powerful.  For
example, many natural equations in pseudo-Riemannian geometry are
projectively invariant, and hence their solutions will be well behaved
toward infinity in the case of projectively compact metrics. 
In contrast to this, such solutions will not be well behaved on, for
example, conformally compact manifolds.  Indeed we exploit such
properties for certan equations in the current work. However in this
article the focus is on further development of the foundational theory
of compactification. In particular we establish fundamental results
linking the asymptotics of the interior geometry to the different
possibilities for projective compactification.

Guided by examples arising from reductions of projective holonomy (see
\cite{ageom,hol-red}), a conceptual approach to projective compactness was
developed in our article \cite{Proj-comp}. The basic idea there was to
start with a linear connection $\nabla$ on $M$ and use local defining
functions for the boundary to define projective modifications of
$\nabla$ which are then required to admit a smooth extension to the
boundary. Applying this to the Levi--Civita connection, the concept
is automatically defined for pseudo--Riemannian metrics. It
turns out that it is natural to involve a real parameter $\al>0$,
called the order of projective compactness. For a local defining
function $\rho$ for the boundary (see Section \ref{2.1} for detailed
definitions) projective compactness of $\nabla$ of order $\al$ then is
the requirement that, if two vector fields $\xi$ and $\eta$ are smooth
up to the boundary, then 
$$
\hat\nabla_\xi\eta=\nabla_\xi\eta+\tfrac1{\al\rho}d\rho(\xi)\eta+
\tfrac1{\al\rho}d\rho(\eta)\xi 
$$ 
admits a smooth extension to the boundary. In the case of
connections preserving a volume density, the order $\alpha$ measures
the growth of that volume density towards the boundary. The main cases of
interest are $\al=1$ and $\al=2$.

In case that $2/\al$ is an integer, the article \cite{Proj-comp}
describes an asymptotic form of a pseudo--Riemannian metric on $M$
which is sufficient for projective compactness of order $\al$. 
In the case $\al=2$, which will be the case of main interest for
this article, this aysmptotic form is given by 
$$
g=\tfrac{h}{\rho}+\tfrac{Cd\rho^2}{\rho^2}.
$$ 
Here $\rho$ is a local defining function for the boundary, $C$ is a
nowhere vanishing function which is smooth up to the boundary and
asymptotically constant in a certain sense, and $h$ is a symmetric
$\binom02$--tensor field, which is smooth up to the boundary and whose
boundary value is non--degenerate in directions tangent to the
boundary.  The metrics introduced by Fefferman and Graham, as
described above, are the class of these with $C=1$.

In this article our first  aim is to treat the more difficult problem of
establishing necessary conditions for projective compactness (of
order 2).
Certain reductions of projective holonomy give rise to examples of
projectively compact connections, and for $\al=2$, the resulting
connections are exactly the Levi--Civita connections of
non--Ricci--flat Einstein metrics. In this case, we have proved in
\cite{Proj-comp} that an asymptotic form as above is always available,
with the function $C$ being a constant related to the Einstein
constant. One of the main results of this article is Theorem
\ref{thm2.5} which shows that the asymptotic form (with constant $C$)
is available for \textit{any} pseudo--Riemannian metric that is
projectively compact of order two. Hence the asymptotic form can be
used as an equivalent definition of projective compactness in this
case.

This result is proved in Section \ref{2}. The main ingredient for the
proof is that for any projectively compact connection $\nabla$ on $M$
the projective structure defined by $\nabla$ admits a smooth extension
to $\barm$. Hence the tools of projective differential geometry, in
particular tractor bundles and tractor connections, all admit smooth
extensions to $\barm$. These can be used to prove that solutions to
certain projectively invariant differential equations automatically
extend smoothly to all of $\barm$. These smooth extension are the main
ingredient in our analysis. In particular, we obtain that the scalar
curvature of $g$ admits a smooth extension to all of $\barm$ and is
asymptotic to a non--zero constant. This nicely complements our result
from \cite{Scalar} that extension of the projective structure of the
Levi--Civita connection of a pseudo--Riemannian metric together with
this type of asymptotic behavior of the scalar curvature forces the
metric to be projectively compact of order two. Thus for
pseudo-Riemmanian metrics with a projective structure that extends to
a boundary at infinity, projective compactness of order 2 (which may
be interpeted as a certain volume growth) is equivalent to the scalar
curvature having a non-zero limit at the boundary. This limit is then
necessarily constant along the boundary. 

The fact that the projective structure defined by a projectively
compact connection or metric extends to all of $\barm$ also gives rise
to an induced geometric structure on the boundary $\partial M$. As a
hypersurface in a projective manifold, $\partial M$ inhertis a
symmetric $\binom02$--tensor field, which is well defined up to
conformal rescaling. This ``projective second fundamental form'' is
the main object of study in Section \ref{3} of this article.

We first show that the projective second fundamental form can be
described in terms of the asymptotic behavior of the Schouten--tensor
(or equivalently the Ricci-tensor), see Proposition
\ref{prop3.1}. This leads to results on the asymptotic behavior of the
curvature of a projectively compact affine connection, as given in
Proposition \ref{prop3.3}. On the other hand, for metrics which are
projectively compact of order two, one can analyse the boundary
geometry in terms of the asymptotic form provided by Theorem
\ref{thm2.5}. In this case we prove, in Proposition \ref{prop3.2},
that the projective second fundamental form is non--degenerate and
thus induces a pseudo--Riemannian conformal structure on $\partial
M$. Together these results lead to a finer description of the
curvature of such a metric. In particular, in Theorem \ref{thm3.3} we
deduce that such a metric satisfies an asymptotic version of the
Einstein equation. Proposition \ref{prop3.2} also proves that for
metrics with the assymptotic form sufficent for projective compactness
of order $\al<2$, the boundary is necessarily totally geodesic; so the
implications of projective compactness for the extrinsic geometry of
the boundary change sharply at $\al=2$.

In Section \ref{4}, we continue the study of the boundary geometry
induced by a connection which is projectively compact of order two,
assuming that the projective second fundamental form is
non--degenerate. By our results, this is always satisfied for
Levi--Civita connections, in general it can be characterized in terms
of the asymptotics of the Schouten tensor. Under these assumptions,
the boundary inherits a pseudo--Riemannian conformal structure, which
can  therefore be described in terms of (conformal) tractors. In the
case of the Levi--Civita connection of a non--Ricci--flat Einstein
metric, one can use the general theory of holonomy reductions to show
that there is a simple relation between conformal tractors on the
boundary and projective tractors in the interior
\cite{ageom,hol-red}. The main aim of Section \ref{4} is to show that,
although the relation is considerably more complicated in general
(which is not surprising in view of the rather intricate relation
between the geometries on $M$ and on $\partial M$), it can still be
described explicitly as follows.

The first main result is that in the general setting it is still the
case that the conformal standard tractor bundle of the boundary may be
naturally identified with the restriction to the boundary of the
projective standard tractor bundle $\mathcal{T}$. This is proved in
Proposition \ref{prop4.1}. This statement requires understanding how
the conformal tractor metric arises. The projectively compact
connection gives rise to a canonical defining density $\tau$ for the
boundary, and applying the BGG splitting operator to this density, one
obtains a bundle metric $L(\tau)$ on the projective standard tractor
bundle. This bundle metric can then be analyzed similarly to the one
obtained from the metricity solution in Section \ref{2}.  We first
prove non--degeneracy of this along $\partial M$, so the boundary
value defines a bundle metric on the restriction of the projective
standard tractor bundle to $\partial M$. This is also established in
Proposition \ref{prop4.1} and it is noted there that using this with
the natural filtration of $\mathcal{T}$ we obtain the usual filtration
of the conformal tractor bundle and tractor metric $L(\tau)$ is seen
to be compatible with conformal metric.

The next main result is Theorem \ref{thm4.1a} which shows that if the
projective tractor covariant derivative of $L(\tau)$ vanishes at the
boundary then the normal conformal tractor connection on the boundary
arises as simply a pullback of the projective tractor connection.

Finally, we treat the general case and describe how the conformal
standard tractor connection on $\partial M$ can be constructed from
the projective standard tractor connection on $\barm$ in two
steps. First one can construct a torsion free tractor connection on
all of $\barm$, which is metric for the given bundle
metric. Restricting this to the boundary, a final step of
normalization leads in Theorem \ref{thm4.4} to the conformal standard
tractor connection. Several simplifications in particular cases (for
example for projectively compact pseudo--Riemannian metrics) are
discussed along the way.

\section{Necessity of asymptotic form}\label{2}
We start by reviewing the concept of projective compactness from
\cite{Proj-comp}, as defined for any affine connection.  Following
this move to the setting where the interior is equipped with a metric.
So given a manifold $\barm$ with boundary $\partial M$ and interior
$M$, we  assume that we have given a pseudo--Riemannian metric $g$
on $M$ such that the projective structure determined by the
Levi--Civita conneciton $\nabla$ of $g$ smoothly extends to
$\barm$. We then specialize to the case that $\nabla$ is
projectively compact of order $\al=2$ and our aim is to prove that this
implies a certain asymptotic form for $g$. The key towards proving
this is to analyze the consequences of the existence of a projectively
compact Levi--Civita connection in the projective class in terms of
tractors.

\subsection{Projective compactness}\label{2.0}
Throughout this article, smooth means $C^\infty$, we consider a smooth
manifold $\barm$ with boundary $\partial M$ and interior $M$. By a
local defining function for $\partial M$ we mean a smooth function
$\rho:U\to [0,\infty)$ defined on an open subset of $\barm$ such that
  $\rho^{-1}(\{0\})=U\cap\partial M$ and $d\rho(x)\neq 0$ for all
  $x\in U\cap\partial M$. By $\Cal E(w)$ we will denote the bundle of
  densities of projective weight $w$. Putting these notions together
leads to the concept
  of a {\em defining density of weight} $w$. We will only need this notion locally. On an open set $U$ of $\barm$, 
this is a  section $\si$
  of $\Cal E(w)$ which is of the form $\rho\hat\si$ for a 
  local defining function $\rho$ for $\partial M$ and a section $\hat\si$ of
  $\Cal E(w)$ which is nowhere vanishing on $U$.

Given an affine connection $\nabla$ and a one--form $\Ups$ on some
manifold, we will write $\hat\nabla=\nabla+\Ups$ for the projectively
modified connection defined by 
\begin{equation}\label{proj-def}
\hat\nabla_\xi\eta=\nabla_\xi\eta+\Ups(\xi)\eta+\Ups(\eta)\xi,
\end{equation}
for vector fields $\xi$ and $\eta$.  Two connections are related in
this way if and only if they have the same geodesics up to
parameterization.

Now in the setting of a manifold with boundary, $\barm=M\cup\partial
M$, a linear connection $\nabla$ on $TM$ is called
\textit{projectively compact of order $\al>0$} if and only if for any
point $x\in\partial M$, there is a local defining function $\rho$ for
$\partial M$ defined on a neighborhood $U$ of $x$, such that the
projectively modified connection
$\hat\nabla=\nabla+\tfrac{d\rho}{\al\rho}$ admits a smooth extension
from $U\cap M$ to all of $U$. This means that for all vector fields
$\xi,\eta$ which are smooth up to the boundary, also
$$
\hat\nabla_\xi\eta=\nabla_\xi\eta+\tfrac1{\al\rho}d\rho(\xi)\eta+
\tfrac1{\al\rho}d\rho(\eta)\xi 
$$ 
admits a smooth extension to the boundary. Equivalently, the
Christoffel symbols of $\hat\nabla$ in some local chart have to admit
such an extension.

It is easily verified that this condition is independent of the choice
of the defining function $\rho$, i.e.~if the projective modification
associated to $\rho$ extends, then also the one associated to any
other defining function is smooth up to the boundary. On the other
hand, the parameter $\al$ cannot be eliminated. Indeed, it turns out
that for connections which are \textit{special}, i.e.~preserve a
volume density, $\al$ controls the growth of a parallel volume density
towards the boundary, see section 2.2 of \cite{Proj-comp}. The result
on volume growth can be nicely reformulated in terms of defining
densities. If $\nabla$ is projectively compact of order $\al$ and
preserves a volume density, then for each $w$, the density bundle
$\Cal E(w)$ admits non--zero parallel sections. However, precisely for
$w=\al$, such a section can be extended by zero to a defining density
for $\partial M$. It is also the case that, for connections preserving
a volume density, projective compactness of order $\al$ is equivalent
to the fact that the projective structure of $\nabla$ extends to all
of $\barm$ plus the appropriate rate of volume growth, see Proposition
2.3 of \cite{Proj-comp}. As in most of \cite{Proj-comp} we will
restrict to the case $0<\al\leq 2$ in this article. For this range of
$\al$ the boundary is at infinity, see Proposition 2.4 in
\cite{Proj-comp}.

\subsection{Metricity of projective structures and tractors}\label{2.1} 
Here and below we use abstract index notation and the convention that
adding ``(w)'' to the name of a vector bundle indicates a tensor
product with the density bundle $\Cal E(w)$. 

Given a smooth manifold of dimension $n+1$ endowed with a projective
structure, one can construct a vector bundle $\Cal T^*$ of rank $n+2$,
which contains the bundle $\Cal E_a(1)$ of weighted one--forms as a
smooth subbundle such that the quotient is isomorphic to $\Cal
E(1)$. This so--called \textit{standard cotractor bundle} can be
canonically endowed with a linear connection $\nabla^{\Cal T^*}$
determined by the projective structure \cite{BEG}. Together, the
configuration of bundle, subbundle and connection is uniquely
determined up to isomorphism. One can then apply constructions with
vector bundles and induced connections to obtain general tractor
bundles, each of which is endowed with a canonical tractor
connection. In particular, the standard tractor bundle $\Cal T$ is the
dual bundle to $\Cal T^*$. We will mainly need the bundles $S^2\Cal
T^*$ and $S^2\Cal T$ of symmetric bilinear forms on $\Cal T$
respectively $\Cal T^*$.

Writing the composition series for $\Cal T^*$ from above as $\Cal
T^*=\Cal E_a(1)\rpl \Cal E(1)$, one can describe the induced
composition series for the other tractor bundles mentioned above as 
\begin{equation}\label{comp-ser}
\begin{aligned}
\Cal T&=\Cal E(-1)\rpl \Cal E^a(-1) \\ S^2\Cal T&=\Cal E(-2)\rpl \Cal
E^a(-2)\rpl\Cal E^{(ab)}(-2) \\ S^2\Cal T^*&=\Cal E_{(ab)}(2)\rpl\Cal
E_a(2)\rpl \Cal E(2).
\end{aligned}
\end{equation}
A choice of connection in the projective class gives rise to an
isomorphism $\Cal T^*\cong\Cal E_a(1)\oplus \Cal E(1)$ and likewise
for the other tractor bundles. Given such a choice, we write sections
of a tractor bundle as column vectors with the component describing the
canonical quotient of the tractor bundle on top and the component in
the canonical subbundle in the bottom. Changing the connection
projectively by a one--form $\Ups_a$, there are explicit formulae for
the changes of these identifications. For the bundles $S^2\Cal T$ and
$S^2\Cal T^*$ we follow the conventions from \cite{Proj-comp}, and the
corresponding formulae for these cases are given as equations (3.5) and
(3.11) in that reference. 

Via the so--called BGG-machinery (see \cite{CSS-BGG},
\cite{Calderbank-Diemer}, and the sketch in \cite{Proj-comp}), each
tractor bundle induces a natural differential operator acting on
sections of its canonical quotient bundle, which defines an
overdetermined system of PDEs (``first BGG--equation'') on that
bundle. Closely related to this is the so--called \textit{splitting
  operator} $L$, which maps sections of the quotient bundle to
sections of the tractor bundle.

We first need the application of these ideas to the metricity equation
for projective structures, see \cite{Eastwood-Matveev}. This
corresponds to the tractor bundle $S^2\Cal T$. Take a projective
manifold $N$ and a pseudo--Riemannian metric $g_{ab}$ on $N$ with
inverse $g^{ab}$. Then $g$ canonically determines a volume density on
$N$, and forming an appropriate power of this density one obtains a
nowhere--vanishing section $\tau\in\Cal E(2)$, which is parallel for
the Levi--Civita connection of $g$. It then turns out that the
Levi--Civita connection of $g$ lies in the given projective class if
and only if $\tau^{-1}g^{ab}$ is a solution of the metricity
equation. The crucial fact for what follows is that this can be
characterized in terms of its image under the splitting
operator. Indeed, in \cite{Eastwood-Matveev}, the authors construct a
natural modification $\nabla^p$ of the tractor connection on $S^2\Cal
T$ such that $\tau^{-1}g^{ab}$ solves the metricity equation if and
only if $L(\tau^{-1}g^{ab})$ is parallel for this modified
connection. In fact, such connections can be constructed for any
tractor bundle (associated to any parabolic geometry), see
\cite{HSSS}.

An explicit formula for the splitting operator $L$ is derived in
Proposition 3.1 of \cite{CGM} (unfortunately with a sign error in the
printed version, that is easily corrected). Given $\si^{ab}\in\Ga(\Cal
E^{(ab)}(-2))$ and a connection $\tilde\nabla$ in the projective class,
the formula for $L(\si^{ab})$ in the splitting determined by
$\tilde\nabla$ on a manifold of dimension $n+1$ is given by
\begin{equation}\label{LStd}
\begin{pmatrix}
\si^{ab} \\ -\frac{1}{n+2}\tilde\nabla_d\si^{dc} \\ 
\frac1{(n+1)(n+2)}\tilde\nabla_d\tilde\nabla_e\si^{de}+
\frac1{n+1}\Rho_{de}\si^{de}\end{pmatrix}.
\end{equation}

Using this, one obtains the following fundamental result that has been
proved in \cite{Scalar}. We include the proof for
completeness. 

\begin{prop}\label{prop2.1}
Let $\barm$ be a smooth manifold with boundary $\partial M$ and
interior $M$. Suppose $g_{ab}$ is a pseudo--Riemannian metric on $M$
such that the projective structure of its Levi--Civita connection
$\nabla$  admits a smooth extension to all of $\barm$. 

Then the corresponding solution $\tau^{-1}g^{ab}\in\Ga(\Cal
E^{(ij)}(-2)|_M)$ of the metricity equation and the scalar curvature
$S\in C^\infty(M,\Bbb R)$ of $g$ both admit smooth extensions to all
of $\barm$.
\end{prop}
\begin{proof}
Since the projective structure of $\nabla$ extends to all of $\barm$,
all projective tractor bundles an tractor connections are defined on
all of $\barm$. The same holds for the modification $\nabla^p$ of the
tractor connection on $S^2\Cal T$ from \cite{Eastwood-Matveev}. 

Now over $M$, we can apply the splitting operator to obtain a section
$L(\tau^{-1}g^{ab})$ of the bundle $S^2\Cal T$. Since
$\tau^{-1}g^{ab}$ satisfies the metricity equation, this section is
parallel for $\nabla^p$ over $M$, so we can extend it by parallel
transport to a smooth section of $S^2\Cal T$ over all of $\barm$,
which is still parallel for $\nabla^p$. Projecting this to a section
of the quotient bundle $\Cal E^{(ab)}(-2)$, we obtain the required
extension of $\tau^{-1}g^{ab}$. 

On the other hand, we can view $L(\tau^{-1}g^{ab})$ as a smooth bundle
metric on $\Cal T^*$ (defined over all of $\barm$). Forming the
determinant of the Gram matrix of this bundle metric with respect to
local frames of $\Cal T^*$ gives rise to a well defined section of the
bundle $(\La^{n+2}\Cal T^*)^2$. Now this bundle is always trivial and
the linear connection inherited from the tractor connection is flat,
so up to an overall non--zero constant factor, this determinant is a
well defined smooth function on $\barm$.

Over $M$, we can work in the splitting determined by $\nabla$. Since
both $g^{ij}$ and $\tau^{-1}$ are parallel for $\nabla$ over $M$, we
conclude from \eqref{LStd} that, over $M$ and in the splitting
corresponding to $\nabla$, we have
\begin{equation}\label{h-nabla}
L(\tau^{-1}g^{ab})=\begin{pmatrix} \tau^{-1}g^{ab} \\ 0
\\ \tfrac{1}{n+1}\tau^{-1}g^{ij}\Rho_{ij}
\end{pmatrix}
\end{equation}
Hence over $M$, the determinant of $L(\tau^{-1}g^{ab})$ is given by
$\tau^{-n-2}\det(g^{ab})g^{ij}\Rho_{ij}$. By definition,
$\tau^{-n-2}\det(g^{ab})=1$, whence $\det(L(\tau^{-1}g^{ab}))$ is a
non--zero multiple of $S$, thus providing the required extension.
\end{proof}

\subsection{The case of projective compactness of order
  two}\label{2.2} 
We next assume that the metric $g_{ab}$ on $M$ is projectively compact
of order two in the sense introduced in \cite{Proj-comp}. By
definition this means that, for any local defining function $\rho$ for
the boundary $\partial M$, we get a distinguished projective
modification ${}^\rho\nabla$ of $\nabla$, which admits a smooth
extension to the boundary. Recall that this modification is associated to the
one--form $\Ups=\tfrac{1}{2\rho}d\rho$, and so for vector fields $\xi$ and
$\eta$ which are smooth up to the boundary also the vector field 
$$
\nabla_\xi\eta+\tfrac{1}{2\rho} d\rho(\xi)\eta+\tfrac{1}{2\rho} d\rho(\eta)\xi
$$
admits a smooth extension of the boundary. 

There is a second crucial consequence of projective compactness of
order two. Namely the non--vanishing density $\tau\in\Ga(\Cal
E(2)|_M)$ can be smoothly extended by zero to all of $\barm$ and then
becomes a defining density for $\partial M$, see Proposition 2.3 of
\cite{Proj-comp}. In terms of a defining function $\rho$ as above,
this means that $\tau=\rho\hat\tau$, where $\hat\tau$ is parallel for
${}^\rho\nabla$ and nowhere vanishing. 

Using these facts we can now analyze the extensions
guaranteed by Proposition \ref{prop2.1}. 

\begin{prop}\label{prop2.2}
In the setting of Proposition \ref{prop2.1}, assume in addition that
$g_{ab}$ is projectively compact of order $2$. Then the zero set of the
boundary value of the smooth extension of $S$ has empty
interior. Hence the boundary value of $L(\tau^{-1}g^{ab})$ is
non--degenerate on a dense open subset of $\partial M$.  
\end{prop}
\begin{proof}
For a local defining function $\rho$ for $\partial M$, we get the
connection ${}^\rho\nabla$, which is smooth up to the boundary. Hence
if we express $L(\tau^{-1}g^{ab})$ over $M$ in the splitting
determined by ${}^\rho\nabla$, Proposition \ref{prop2.1} implies that
all slots will admit a smooth extension to the boundary. We use
Formula \eqref{h-nabla} from the proof of Proposition \ref{prop2.1}
together with the formula (3.11) in Section 3.6 of \cite{Proj-comp} and
the fact that $\Ups_a=\tfrac{1}{2\rho}\rho_a$, where
$\rho_a=d\rho$. This shows that over $M$ in the splitting determined
by ${}^\rho\nabla$, we get 
\begin{equation}\label{h-nrho}
L(\tau^{-1}g^{ab})=\begin{pmatrix}  \tau^{-1}g^{ab}
\\ \tfrac{-1}{2\rho}\tau^{-1}g^{ci}\rho_i 
\\ \tfrac{1}{n+1}\tau^{-1}g^{ij}\Rho_{ij}+
\tfrac{1}{4\rho^2}\tau^{-1}g^{ij}\rho_i\rho_j 
\end{pmatrix}.
\end{equation}
Since $\tau^{-1}=\tfrac{1}{\rho}\hat\tau^{-1}$ and $\hat\tau^{-1}$ is
nowhere vanishing, we conclude that $\tfrac{1}{\rho}g^{ab}$,
$\tfrac{1}{\rho^2}g^{ai}\rho_i$ and 
$$
\tfrac{1}{\rho}\left(\tfrac{1}{n+1}g^{ij}\Rho_{ij}+
\tfrac{1}{4\rho^2}g^{ij}\rho_i\rho_j\right)
$$ 
admit smooth extensions to the boundary. But from Proposition
\ref{prop2.1} and the fact that $\tfrac{1}{\rho^2}g^{ai}\rho_i$
extends, we conclude that already the sum in the bracket in the last
displayed formula admits a smooth extension to the boundary. Thus we
conclude that this extension has to vanish along the boundary, so
$\tfrac{1}{4\rho^2}g^{ij}\rho_i\rho_j$ approaches
  $\tfrac{-1}{n+1}g^{ij}\Rho_{ij}$ at the boundary. 

We can phrase the information we have obtained so far in terms of the
matrix expression for the inverse metric $g^{ab}$ with respect to  a local
frame for $T^*M$ which is smooth up to the boundary and has $\rho_a$
as its first element. Since $\tfrac{1}{\rho^2}g^{ai}\rho_i$ admits a
smooth extension, the elements in the first row and in the first
column of this matrix can be written as $\rho^2a^{1j}$ (respectively
$\rho^2a^{j1}$) for functions $a^{1j}=a^{j1}$ which are smooth up to
the boundary. On the other hand, since $\tfrac{1}{\rho}g^{ab}$ admits
a smooth extension, the other entries in the matrix can all be written
as $\rho a^{ij}$, where again $a^{ij}=a^{ji}$ is smooth up to the
boundary. 

With a view towards contradiction, assume that the boundary value of
$S$ vanishes on an open subset of $\partial M$. Restricting to an
appropriate open subset in $\barm$, we can assume that $S$ vanishes
identically. Thus $g^{ij}\Rho_{ij}$ vanishes along the boundary, so
from above we conclude that $\tfrac{1}{\rho^2}g^{ij}\rho_i\rho_j$
vanishes along the boundary. Thus we can write $g^{ij}\rho_i\rho_i$ as
$\rho^3 a^{11}$ where $a^{11}$ is smooth up to the boundary. Forming
$\det(g^{ij})$ we see that we can first take a factor $\rho^2$ out of
the first row and then a factor $\rho$ out of each of the other rows,
and then still a factor $\rho$ out of the first column, so
$\det(g^{ij})=\rho^{2+n+1}\det(a^{ij})$. On the other hand, the
results on volume asymptotics from Proposition 2.3 of \cite{Proj-comp}
show that, viewed as a density, $\det(g_{ij})=\rho^{-n-2}\hat\nu$
where $\hat\nu$ is smooth up to the boundary. Multiplying we see that
$1=\rho\det(a_{ij})$ which contradicts the fact that $\det(a^{ij})$ is
smooth up to the boundary.
\end{proof}

\subsection{The inverse of the metricity tractor}\label{2.3}
Knowing that the boundary value of $S$ is non--zero on a dense open
subset of $\partial M$, we know that the metricity tractor
$L(\tau^{-1}g^{ab})$ is a non--degenerate bundle metric on a
neighborhood of this subset in $\barm$. Hence we can form its inverse
there, and this is a section $\Ph$ of $S^2\Cal T^*$ which is smooth up
to the boundary and non--degenerate. This leads to fundamental
information on the asymptotic behavior of the metric $g_{ab}$.

\begin{prop}\label{prop2.3}
Let $g$ be a pseudo--Riemannian metric on $M$ which is projectively
compact of order $2$.  Let $x\in\partial M$ be a point such that the
boundary value of the scalar curvature $S$ is non--zero in $x$, and
let $\rho$ be a defining function for $\partial M$ which is defined on
a neighborhood of $x$ in $\barm$.

Then locally around $x$, the section 
$$ h_{ab}:=\rho
g_{ab}+\tfrac{n+1}{4\rho}(g^{ij}\Rho_{ij})^{-1}\rho_a\rho_b
$$
admits a smooth extension to the boundary and the boundary values are
non--degenerate as bilinear forms on the spaces $T_y\partial M$. 
\end{prop}
\begin{proof}
Restricting to an appropriate open neighborhood $U$ of $x$, we may
assume that $L(\tau^{-1}g^{ab})\in\Ga(S^2\Cal T)$ is non--degenerate
as a bundle metric on $\Cal T^*$ and hence its inverse
$\Ph\in\Ga(S^2\Cal T^*)$ is a smooth non--degenerate bundle metric on
$\Cal T$. On $U\cap M$, we can work in the splitting associated to the Levi-Civita connection
$\nabla$ of $g$. The expression \eqref{h-nabla} for $L(\tau^{-1}g^{ab})$ from
the proof of Proposition \ref{prop2.1} then implies that in that
splitting we get 
$$
\Ph=\begin{pmatrix} (n+1)\tau (g^{ij}\Rho_{ij})^{-1} \\ 0\\ \tau
g_{ab}\end{pmatrix}.
$$ As in the proof of Proposition \ref{prop2.2}, we can now compute
the expression for $\Ph$ in the splitting associated to the connection
${}^\rho\nabla$, which is defined up to the boundary. All slots in
this expression then  admit smooth extensions to the boundary
and $\Ph$ must be non--degenerate, also along the boundary. We can
compute the change of splitting from formula (3.5) in Section 3.1 of
\cite{Proj-comp} using that $\Ups_a=\tfrac{1}{2\rho}\rho_a$. This
shows that in the splitting associated to ${}^\rho\nabla$, we get
\begin{equation}\label{Phi-nrho}
\Ph=\begin{pmatrix} \hat\tau\rho(n+1)(g^{ij}\Rho_{ij})^{-1}
\\ \hat\tau \tfrac{n+1}2 (g^{ij}\Rho_{ij})^{-1}\rho_a \\ 
\hat\tau(\rho
g_{ab}+\tfrac{n+1}{4\rho}(g^{ij}\Rho_{ij})^{-1}\rho_a\rho_b)
\end{pmatrix}.
\end{equation}
The bottom slot is $\hat\tau h_{ab}$, so we see that $h_{ab}$ admits a
smooth extension to the boundary. Along the boundary, the top slot
vanishes, while the middle slot becomes a non--zero multiple of
$\hat\tau\rho_a$. Non--degeneracy of $\Ph$ along the boundary is then
equivalent to the fact that $h_{ab}$ is non--degenerate on the kernel
of the middle slot, which coincides with $T\partial M\subset
TM|_{\partial M}$.  
\end{proof}

\subsection{Geodetic transversals}\label{2.4}
We next discuss a natural product structure along the boundary. This
can be done for affine connections which are projectively compact of
arbitrary order, so we temporarily work in this more general setting.

Suppose that $\nabla$ is a linear connection on $TM$, which is
projectively compact of some order $\al>0$, and that $\rho$ is a local
defining function for the boundary $\partial M$, defined on an open set
$U\subset\barm$. Then by definition the affine connection
$\nabla+\frac{d\rho}{\al\rho}$ defined on $U\cap M$ extends smoothly
to all of $U$, and we again denote this connection by ${}^\rho\nabla$.

\begin{definition}\label{def2.4}
A \textit{geodetic transversal} for $\rho$ is a smooth vector field
$\mu\in\frak X(U)$ such that ${}^\rho\nabla_\mu\mu=0$ (i.e.~the flow
lines of $\mu$ are geodesics for ${}^\rho\nabla$) and such that
$d\rho(\mu)$ is identically one on $U\cap\partial M$.
\end{definition}

\begin{lemma}\label{lem2.4}
(i) Given $U$ and $\rho$ and a vector field $\mu_0$ along
  $U\cap\partial M$ such that $d\rho(\mu_0)=1$ on $U\cap\partial M$,
  we can (possibly shrinking $U$) extend $\mu_0$ uniquely to a
  geodetic transversal for $\rho$.

(ii) If $\mu$ is any geodetic transversal for $\rho$ then for each
  point $x\in U\cap\partial M$ there is an open neighborhood $\tilde
  V$ of $x$ in $\barm$, a positive $\ep\in\Bbb R$, and a
  diffeomorphism $\tilde V\to [0,\ep)\x V$ where $V=\tilde
    V\cap\partial M$, which maps each $y\in V$ to $(0,y)$, and pulls
    back the coordinate vector field $\partial_t$ for the coordinate
    $t$ in $[0,\ep)$ to $\mu$.
\end{lemma}
\begin{proof}
Extend $\mu_0$ to a local smooth frame for $T\barm|_{\partial
  M}$. Denoting by $p:\Cal P\barm\to\barm$ the linear frame bundle of
$\barm$ and by $\th\in\Om^1(\Cal P\barm,\Bbb R^{n+1})$ its soldering
form, the frame defines a smooth map $s:U\cap\partial M\to \Cal
P\barm$ such that $p\o s=\id$.

Now ${}^\rho\nabla$ defines a principal connection on $\Cal P\barm$,
so we can talk about horizontal vector fields on $\Cal P\barm$ and
such a field is uniquely determined by its value under $\th$. In
particular, let $X\in\frak X(\Cal P\barm)$ be the horizontal vector
field whose value under $\th$ is always the first vector in the
standard basis of $\Bbb R^{n+1}$. This means that for any frame
$u\in\Cal P\barm$, $T_up\cdot X(u)$ is the first element in the frame
$u$, so in particular $T_{s(y)}p\cdot X(s(y))=\mu_0(y)$ for all $y\in
U\cap\partial M$.

Let us denote by $\Fl^X_t$ the flow of the vector field $X$, and
consider the map $(y,t)\mapsto p(\Fl^X_t(s(y)))$, which is defined and
smooth on an open neighborhood of $(U\cap\partial M)\x\{0\}$ in
$(U\cap\partial M)\x [0,\infty)$. Evidently, its tangent map in
  $(y,0)$ restricts to the identity on $T_y\partial M$ and maps
  $\partial_t$ to $\mu_0(y)$ so it is a linear isomorphism. Hence for
  any $y\in U\cap\partial M$, it restricts to a diffeomorphism on a
  set of the form $V\x [0,\ep)$ where $V\subset\partial M$ is an open
    neighborhood of $y$. Since the flow lines of $X$ in $\Cal P\barm$
    project to geodesics in $M$, we can define $\mu$ as the image of
    $\partial_t$ under this diffeomorphism to complete the proof of
    (i), and use the inverse of the diffeomorphism to complete the
    proof of (ii).
\end{proof}

\subsection{The asymptotic form}\label{2.5}
Returning to the setting of a pseudo--Riemannian metric $g$ which is
projectively compact of order $2$, we can next use a geodetic
transversal to complete the description of the asymptotic behavior of
the scalar curvature $S$ of $g$.

\begin{prop}\label{prop2.5}
Let $g$ be a pseudo--Riemannian metric on $M$ with scalar curvature
$S$ which is projectively compact of order $2$.  Let $x\in\partial M$
be a point in which the boundary value of $S$ is non--vanishing, let
$\rho$ be a local defining function for $\partial M$ which is defined
on an open neighborhood $U$ of $x\in\barm$ and suppose that $\mu$ is a
geodetic transversal for $\rho$ defined on $U$. Then we have:

(1) The function $\rho^2g(\mu,\mu)$ is constant along flow lines of
$\mu$ and hence admits a smooth extension to the boundary. The
boundary value of this extension equals the one of
$-\tfrac{n+1}4(g^{ij}\Rho_{ij})^{-1}$.

(2) The boundary value of $\rho^2g(\mu,\mu)$ is constant on a
neighborhood of $x$. 
\end{prop}
\begin{proof}
(1) On $U\cap M$, we compute 
$$
\mu\cdot (\rho^2g(\mu,\mu))=2\rho
d\rho(\mu)g(\mu,\mu)+2\rho^2g(\nabla_\mu\mu,\mu). 
$$
We can write
$\nabla_\mu\mu={}^\rho\nabla_\mu\mu-2\Ups(\mu)\mu=-\tfrac{1}{\rho}d\rho(\mu)\mu$,
and inserting this, we see that $\rho^2g(\mu,\mu)$ is constant along
flow lines of $\mu$. 

From Proposition \ref{prop2.3}, we know that $\rho
g_{ab}+\tfrac{n+1}{4\rho}(g^{ij}\Rho_{ij})^{-1} \rho_a\rho_b$ admits a smooth extension to
the boundary. Multiplying this by $\rho$, we obtain a tensor field
which is smooth up to the boundary and vanishes along the
boundary. Inserting two copies of $\mu$ into this tensor field we see
that 
$$
\rho^2g(\mu,\mu)+\tfrac{n+1}{4}(g^{ij}\Rho_{ij})^{-1}(d\rho(\mu))^2
$$
approaches zero at the boundary, and since $d\rho(\mu)$ equals one
along the boundary, the proof of (1) is complete. 

\smallskip

(2) Let $\xi=\xi^a$ be a vector field on $U$ (so $\xi$ is smooth up to
the boundary), such that $d\rho(\xi)$ vanishes identically. Then
Proposition \ref{prop2.3} immediately implies that for any vector field
$\eta\in\frak X(U)$, the function $\rho g(\xi,\eta)$ admits a smooth
extension to the boundary. Next, since $d\rho(\mu)$ equals one along
$\partial M$ and $\xi$ is tangent to $\partial M$ along $\partial M$,
so $\xi\cdot (d\rho(\mu))$ vanishes along $\partial
M$. Expanding $0=d(d\rho)(\xi,\mu)$ and using that $d\rho(\xi)=0$ we
conclude that $d\rho([\xi,\mu])$ vanishes along $\partial M$. Again
using Proposition \ref{prop2.3}, we conclude that also $\rho
g([\xi,\mu],\eta)$ admits a smooth extension to the boundary for each
$\eta\in\frak X(U)$. Armed with these observations, we now compute
$$
\xi\cdot (\rho^2g(\mu,\mu))=\rho^2\xi\cdot g(\mu,\mu)=2\rho^2
g(\nabla_\xi\mu,\mu)=2\rho^2g([\xi,\mu],\mu)+2\rho^2g(\nabla_\mu\xi,\mu). 
$$
From above we see that the first term in the right hand side admits a
smooth extension to the boundary with boundary value zero. The second
term on the right hand side can be written as
$$ 
2\rho^2 \mu\cdot g(\xi,\mu)-2\rho^2 g(\xi,\nabla_\mu\mu)=2\rho^2
\mu\cdot g(\xi,\mu)+2\rho d\rho(\mu)g(\xi,\mu),
$$ where we have used the expression for $\nabla_\mu\mu$ obtained
above. Rewriting the right hand side as $2\rho\mu\cdot (\rho
g(\xi,\mu))$ we see that also this terms admits a smooth extension to
the boundary with boundary value zero. Hence $\xi\cdot
(\rho^2g(\mu,\mu))$ vanishes along the boundary, and since we can
realize any vector field tangent to the boundary as a boundary value
in this way, we see that $\rho^2g(\mu,\mu)$ is locally constant along
the boundary.
\end{proof}

From this, we can readily deduce our first main result. 

\begin{thm}\label{thm2.5}
Let $g$ be a pseudo--Riemannian metric on $M$, which is projectively
compact of order $2$. Then we have

(1) The smooth extension $S$ of the scalar curvature of $g$ to all of
$\barm$ guaranteed by Proposition \ref{prop2.1} has a boundary value
which is locally constant and nowhere vanishing.

(2) Given a boundary point $x\in\partial M$ and a local defining
function $\rho$ for $\partial M$, then for the non--zero constant
$C=\tfrac{-n(n+1)}{4S(x)}$, the tensor field
$$
h_{ab}:=\rho g_{ab} - \tfrac{C}{\rho}\rho_a\rho_b
$$ admits a smooth extension to the boundary with its boundary values
being non--degenerate as bilinear forms on $T\partial M$.
\end{thm}
\begin{proof}
(1) From Propositions \ref{prop2.1} and \ref{prop2.5}, we know that
  $S$ is smooth up to the boundary, and that $S|_{\partial M}$ is
  non--vanishing and locally constant on a dense open subset of
  $\partial M$. This is only possible if the constant values on
  connected components of this open set with intersecting closures
  match up, and hence $S$ extends to a locally constant function on
  $\partial M$. But of course all the constant values are non--zero,
  so $S$ is nowhere vanishing. 

(2) In our convention for dimensions, we have
  $\Rho_{ab}=\tfrac1n R_{ab}$, where $R_{ab}:=R_{d a}{}^d{}_b$ (is the Ricci
curvature) and hence
  $g^{ij}\Rho_{ij}=\tfrac{1}nS$, so the claim follows immediately from
  Proposition \ref{prop2.3}. 
\end{proof}

\section{Boundary geometry and curvature asymptotics}\label{3}
A projectively compact connection on the interior of a manifold with
boundary induces a projective structure on the whole manifold. As a
hypersurface in a projective manifold, the boundary inherits the
so--called projective second fundamental form, a conformal class of
bilinear forms on the tangent spaces to the boundary. After making
these observations, our main aim in this section is to relate this
structure on the boundary to data on the interior. We first do this
for general projectively compact affine connections, showing that the
projective second fundamental form is related to the asymptotics of
the Schouten tensor. This leads to results on the asymptotic form of
the curvature of a projectively compact connection.

With these results established, we then turn our attention to
pseudo--Riemannian metrics admitting an asymptotic form from a family
(depending on $\al\in (0,2]$) identified in \cite{Proj-comp}; in that
  source it is shown that these asymptotic forms are sufficient for
  projective compactness. We find that the possible extrinsic boundary
  geometry depends on the parameter $\alpha$, which gives the order of
  projective compactness.  If the order of projective compactness is
  less than two, then the projective second fundamental form
  necessarily vanishes, so the boundary is totally geodesic. On the
  other hand, in the case of order two there is no such a prior
  restriction on the projective second fundamental form and we obtain
  an explicit description of this object.  Note that Theorem
  \ref{thm2.5} states that for metrics that are projectively compact
  of order 2 the asymptotic form is always available. So the results
  apply generally in this case.
The explicit formula description found implies, in particular, that the projective second
  fundamental form is always non--degenerate, in the order two case, and hence defines a
  canonical conformal structure on the boundary. Finally, the relation
  to the Schouten tensor is used to prove that any such metric
  satisfies an asymptotic form of the Einstein equation.

\subsection{The induced geometry on the boundary}\label{3.1}
Suppose that $\barm=M\cup\partial M$ is a smooth manifold with
boundary and that $\nabla$ is an affine connection on $M$ which is
projective compact of some order $\al\in (0,2]$. Let us recall the
construction of the projectively invariant second fundamental form for
the extended projective structure.

Choose a local defining function $\rho$ for $\partial M$, let
$\hat\nabla$ be any connection in the projective class which is smooth
up to the boundary and consider $\hat\nabla
d\rho\in\Ga(S^2T\barm)$. Writing again $\rho_a$ for $d\rho$, we see
that for a projectively equivalent connection $\tilde\nabla$, we get
$\tilde\nabla_a\rho_b=\hat\nabla_a\rho_b-\Ups_a\rho_b-\Ups_b\rho_a$,
so $\hat\nabla_a\rho_b$ and $\tilde\nabla_a\rho_b$ have the same
restriction to $T\partial M\x T\partial M$. On the other hand,
changing the defining function $\rho$ to $\tilde\rho=e^f\rho$, we get
$\tilde\rho_a=\tilde\rho f_a+e^f\rho_a$, where $f_a=df$, and thus
$$
\hat\nabla_a\tilde\rho_b=\tilde\rho_af_b+\tilde\rho\hat\nabla_af_b+
e^ff_a\rho_b+e^f\hat\nabla_a\rho_b. 
$$ 
Hence the restriction of $\hat\nabla_a\tilde\rho_b$ to $T\partial M\x
T\partial M$ is conformal to the restriction of $\hat\nabla_a\rho_b$.
So the (possibly degenerate) conformal class $[\hat\nabla_a\rho_b]$ on
$T\partial M$ is canonical. We will say that   
any positive constant multiple of $\hat\nabla_a\rho_b$ is  a
representative of the \textit{projective second fundamental form}.

By construction, the projective second fundamental form only depends
on the extended projective structure on the manifold $\barm$ with
boundary, and not on the specific projectively compact connection
$\nabla$ on $M$.  It turns out, however, that there is a nice relation
to the projectively compact connection on the interior.

\begin{prop}\label{prop3.1}
Let $\nabla$ be a linear connection on $TM$ which is projectively
compact of some order $\al\in (0,2]$, and let $\Rho_{ab}$ be the
  Schouten tensor of $\nabla$.

Then, for any local defining function $\rho$ for $\partial M$ the
smooth section
$\rho\Rho_{ab}+\frac{\al-1}{\al^2}\frac{\rho_a\rho_b}{\rho}$ admits a
smooth extension to the boundary and its boundary value restricts to a
representative of the projective second fundamental form on $T\partial
M$.
\end{prop}
\begin{proof}
Let $\hat\nabla={}^\rho\nabla$ be the projective modification of
$\nabla$ associated to $\rho$. This means that
$\hat\nabla_a=\nabla_a+\Ups_a$, with $\Ups_a:=\frac{\rho_a}{\al\rho}$,
admits a smooth extension to the boundary. Then of course the Schouten
tensor $\hat\Rho_{ab}$ of $\hat\nabla$ is smooth up to the
boundary. The relation between $\Rho_{ab}$ and $\hat\Rho_{ab}$ from
\cite{BEG} reads as
$$
\Rho_{ab}=\hat\Rho_{ab}+\hat\nabla_a\Ups_b+\Ups_a\Ups_b.
$$
Now
$\hat\nabla_a(\tfrac1{\al\rho}\rho_b)=-\tfrac1{\al\rho^2}\rho_a\rho_b+
\frac1{\al\rho}\hat\nabla_a\rho_b$. On the other hand,
$\Ups_a\Ups_b=\frac1{\al^2\rho^2}\rho_a\rho_b$, and inserting this, we
get that 
$$
\Rho_{ab}=\hat\Rho_{ab}+\tfrac1{\al\rho}\hat\nabla_a\rho_b-
\tfrac{\al-1}{\al^2\rho^2}\rho_a\rho_b
$$
and thus 
\begin{equation}\label{Rhoasymp}
\rho\Rho_{ab}+\tfrac{\al-1}{\al^2\rho}\rho_a\rho_b=\tfrac1{\al}\hat\nabla_a\rho_b+
\rho\hat\Rho_{ab}. 
\end{equation}
Since, the right hand side is evidently smooth up to the boundary, with
boundary value $\tfrac1{\al}\hat\nabla_a\rho_b$, the result follows.
\end{proof}

\subsection{Curvature asymptotics}\label{3.3}
We next prove a general result on the asymptotic behavior of the
curvature of a connection which is projectively compact of some order
$\al\in (0,2]$. This is similar to the fact that conformally compact
    pseudo--Riemannian metrics are asymptotically hyperbolic, see
for example    \cite{Graham:Srni}.

To formulate the result, recall that from each symmetric
$\binom02$--tensor field, one can build up a tensor having curvature
symmetries by putting
$R_{ab}{}^c{}_d:=\delta^c_a\ph_{bd}-\delta^c_b\ph_{ad}$. In
particular, we can apply this to $\binom02$--tensor fields which have
rank one, i.e.~are of the form $\ph_{ab}=\ps_a\ps_b$ for a one--form
$\ps=\ps_a$. In this case we call the corresponding curvature tensor
the \textit{rank--one curvature tensor determined by} $\ps$.

\begin{prop}\label{prop3.3}
Let $\nabla$ be a linear connection on $TM$ which is projectively
compact of some order $\al\in (0,2]$, let $R=R_{ab}{}^c{}_d$ be
  the curvature tensor of $\nabla$. Let $\rho$ be a local defining
  function for $\partial M$ and let
  $\hat\nabla=\nabla+\tfrac{d\rho}{\al\rho}$ be the associated
  connection in the projective class.  

(i) If $\al=1$, then $\rho R$ is admits a smooth extension to the
  boundary with boundary value
$$ 
\delta^c_a\hat\nabla_b\rho_d-\delta^c_b\hat\nabla_a\rho_d.
$$ 

(ii) If $\al\neq 1$, then $\rho^2R$ admits a smooth extension to the
boundary with boundary value equal to $\frac{1-\al}{\al^2}$ times the
rank--one curvature tensor determined by the one--form $d\rho$.
\end{prop}
\begin{proof}
The decomposition of the curvature tensor used in projective geometry,
see section 3.1 of \cite{BEG}, reads as 
$$ 
R_{ab}{}^c{}_d=C_{ab}{}^c{}_d+\delta^c_a\Rho_{bd}-\delta^c_b\Rho_{ad}+
\be_{ab}\delta^c_d. 
$$ Here $C_{ab}{}^c{}_d$ is the projective Weyl curvature, $\Rho_{ab}$
is the projective Schouten tensor and $\be_{ab}=\Rho_{ba}-\Rho_{ab}$
(so this vanishes for connections preserving a volume density). Now
the projective Weyl curvature is projectively invariant, so since the
projective structure extends smoothly to $\barm$, $C_{ab}{}^c{}_d$
admits a smooth extension to the boundary.

We have analyzed the behavior of $\Rho_{ab}$ in Proposition
\ref{prop3.1}. If $\al=1$, then
$\rho\Rho_{ab}=\hat\nabla_a\rho_b+\rho\hat\Rho_{ab}$, where
$\hat\Rho_{ab}$ is the Schouten tensor of $\hat\nabla_a$. Of course
$\hat\Rho_{ab}$ is smooth up to the boundary, and we conclude that,
$\hat\be_{ab}=\be_{ab}$, so $\be_{ab}$ is smooth up to the
boundary. This completes the proof of (i).

(ii) If $\al\neq 1$, then Proposition \ref{prop3.1} shows that
$\rho^2\Rho_{ab}$ admits a smooth extension to the boundary with
boundary value $\frac{1-\al}{\al^2}\rho_a\rho_b$, so again the result
follows.
\end{proof}

\subsection{Projectively compact pseudo--Riemannian metrics and
  asymptotic forms}\label{3.4} 

The asymptotic form for a metric which
is projectively compact of order two derived in Theorem
\ref{thm2.5} is a special case of an asymptotic form (depending on
$\al$) introduced in Section 2.4 of \cite{Proj-comp}. There we have
proved that such an asymptotic form for $g$ implies projective
compactness of order $\al$ for any fixed $\al\in (0,2]$ such that
$\frac2\al$ is an integer. We next specialize the results on
the boundary conformal structure and on curvature asymptotics to
metrics admitting such an asymptotic form. 

The assumptions for this asymptotic form is that locally around each
boundary point, we find a defining function $\rho$ and a nowhere
vanishing smooth function $C$ with additional properties specified
below, such that the $\binom02$--tensor field
\begin{equation}\label{haldef}
h:=\rho^{2/\al} g-C\frac{d\rho^2}{\rho^{2/\al}}
\end{equation}
admits a smooth extension to the boundary, with the boundary value
being non--degenerate on $T\partial M$. The additional property
required from $C$ is that for each vector field $\zeta$ which is
smooth up to the boundary and satisfies $d\rho(\zeta)=0$, the function
$\rho^{-2/\al}\zeta\cdot C$ admits a smooth extension to the boundary.
Theorem 2.6 of \cite{Proj-comp} then states that under these
assumptions (including $2/\al\in\Bbb Z$), the Levi--Civita connection
of $g$ is projectively compact of order $\al$. For $\al=2$, Theorem
\ref{thm2.5} shows that we always get the asymptotic form with a
constant $C$, so we will pay special attention to this case.

\begin{prop}\label{prop3.2}
Suppose that we are in the setting of Theorem 2.6 of \cite{Proj-comp},
i.e.~$\frac2\al\in\Bbb Z$ and the tensor field $h$ defined in
\eqref{haldef} is smooth up to the boundary with the boundary value
being non--degenerate on $T\partial M$.

(i) If $\al<2$, then the projective second fundamental form for
$\partial M$ vanishes identically, so $\partial M$ is totally
geodesic. 

(ii) If $\al=2$, then the restriction of $h$ to boundary directions is
a representative of the projective second fundamental form for
$\partial M$. Furthermore, if $C$ is constant then the boundary value
of $h$ coincides with $-2C \hat\nabla d\rho$, where
$\hat\nabla=\nabla+\frac{d\rho}{2\rho}$ is the projective
modification, associated to $\rho$, of the Levi--Civita connection
$\nabla$ of $g$.
\end{prop}
\begin{proof}
We use ideas from the proof of Theorem 2.6 of \cite{Proj-comp} and
also the notation introduced there.  In the proof of that theorem, one
first constructs a vector field $\zeta_0$ such that
$d\rho(\ze_0)\equiv 1$ and $\ze_0$ is orthogonal with respect to $h$
to all vector fields in the kernel of $d\rho$. In particular, as
observed there, for any tangent vector fields $\xi$, $\eta$  one can compute the boundary value of
$-d\rho(\hat\nabla_\xi\eta)$ as the boundary value of
$\tfrac{-1}{C}\rho^{4/\al}g(\hat\nabla_\xi\eta,\zeta_0)$. 

A key ingredient in the proof of Theorem 2.6 of \cite{Proj-comp} is
the modified Koszul formula, which says that
$2g(\hat\nabla_\xi\eta,\zeta_0)$ can be computed as
\begin{equation}\label{modKos}
\begin{aligned}
&\xi\cdot g(\eta,\ze_0)-\ze_0\cdot g(\xi,\eta)+\eta\cdot
  g(\xi,\ze_0)+g([\xi,\eta],\ze_0)-g([\xi,\ze_0],\eta)\\ 
-&g([\eta,\ze_0],\xi)+\tfrac{2d\rho(\xi)}{\al\rho}g(\eta,\ze_0)+
  \tfrac{2d\rho(\eta)}{\al\rho}g(\xi,\ze_0).
\end{aligned}
\end{equation}
Let us first assume that $d\rho(\xi)=d\rho(\eta)=0$. Then we get
$d\rho([\xi,\eta])=-dd\rho(\xi,\eta)=0$, so $g(\xi,\zeta_0)$,
$g(\eta,\ze_0)$ and $g([\xi,\eta],\ze_0)$ vanish identically. Next,
$g([\xi,\zeta_0],\eta)=\frac1{\rho^{2/\al}}h([\xi,\zeta_0],\eta)$, so
after multiplication by $\rho^{4/\al}$ this extends smoothly to the
boundary by zero, and the same holds for the corresponding term with
$\xi$ and $\eta$ exchanged. In conclusion, we see that we can compute the
boundary value of $-d\rho(\hat\nabla_\xi\eta)$ as the boundary value
of
$$
\tfrac{1}{2C}\rho^{4/\al}\zeta_0\cdot
g(\xi,\eta)=\tfrac{1}{2C}\rho^{4/\al}\zeta_0\cdot
\tfrac{1}{\rho^{2/\al}}h(\xi,\eta). 
$$ Up to terms vanishing along the boundary, this equals
$\tfrac{-1}{\al C}\rho^{(2-\al)/\al}h(\xi,\eta)$. But since
$d\rho(\eta)=0$, we get $-d\rho(\hat\nabla_\xi\eta)=(\hat\nabla_\xi
d\rho)(\eta)$, so we get (i) and the first part of (ii).

\smallskip

To obtain the second statement in (ii) we have to analyze (in the case
$\al=2$ and for $C$ being constant) the modified Koszul formula
\eqref{modKos} for general vector fields $\xi$ and $\eta$, which needs
much more care. From the proof of Theorem 2.6 in \cite{Proj-comp} we see
that (always taking into account that $\al=2$) 
\begin{gather*}
g(\eta,\ze_0)=d\rho(\eta)(\tfrac{C}{\rho^2}+\tfrac1\rho
h(\ze_0,\ze_0))\\ 
g(\xi,\eta)=\tfrac{C}{\rho^2}d\rho(\xi)d\rho(\eta)+\tfrac1\rho
h(\xi,\eta). 
\end{gather*}
Now if we plug the appropriate versions of these into the modified
Koszul formula \eqref{modKos} and carry out the differentiations, we
can sort the terms according to powers of $\rho$. In the proof of
Theorem 2.6 of \cite{Proj-comp} it is shown that the terms containing
$\tfrac{1}{\rho^3}$ add up to zero. We have to determine the terms
containing $\tfrac{1}{\rho^2}$ while we may ignore terms containing
$\tfrac{1}{\rho}$ or no negative power of $\rho$. The first and third
term in \eqref{modKos} together contribute
\begin{equation}\label{tech1}
C\xi\cdot d\rho(\eta)+C\eta\cdot
d\rho(\xi)-2d\rho(\xi)d\rho(\eta)h(\ze_0,\ze_0) 
\end{equation}
to the coefficient of $\tfrac{1}{\rho^2}$. Now the last part of this
cancels with the contribution of the last two summands in
\eqref{modKos}. On the other hand, the only contribution of the fourth
summand in \eqref{modKos} to the coefficient of $\tfrac{1}{\rho^2}$ is
$Cd\rho([\xi,\eta])$. Expanding $0=dd\rho(\xi,\eta)$ we see that this
adds up with the second term in \eqref{tech1} to $C\xi\cdot
d\rho(\eta)$, so the overall contribution of all terms we have
considered so far is $2C\xi\cdot d\rho(\eta)$. 

Next, the contribution of the second summand of \eqref{modKos} to the
coefficient of $\tfrac{1}{\rho^2}$ is given by 
$$
h(\xi,\eta)-C\ze_0\cdot(d\rho(\xi)d\rho(\eta)), 
$$
while the fifth and sixth summands contribute 
$$
-Cd\rho([\xi,\ze_0])d\rho(\eta)-Cd\rho([\eta,\ze_0])d\rho(\xi). 
$$ 
But since $d\rho(\ze_0)\equiv 1$, the fact that $0=dd\rho(\xi,\ze_0)$
implies that $\ze_0\cdot d\rho(\xi)=d\rho([\xi,\ze_0])$ and likewise
for $\eta$, so these terms together only contribute $h(\xi,\eta)$. 

Collecting the results, we see that the boundary value of
$-d\rho(\hat\nabla_\xi\eta)$ can be computed as the boundary value of
$\tfrac{-1}{2C}(2C\xi\cdot d\rho(\eta)+h(\xi,\eta))$. Bringing the
first term to the other side, we obtain the boundary value of
$(\hat\nabla d\rho)(\xi,\eta)$ which implies the result. 
\end{proof}

Next, we describe the curvature for pseudo--Riemannian metrics which
are projectively compact of order two and show that they satisfy an
asymptotic version of the Einstein equation.

\begin{thm}\label{thm3.3}
Let $g=g_{ab}$ be a pseudo--Riemannian metric on $M$, with inverse
$g^{ab}$, which is projectively compact or order two and let
$h=h_{ab}$ and $C$ be as in Theorem \ref{thm2.5}. Let $R_{ab}{}^c{}_d$
be the Riemann curvature of $g$, $R_{ab}=R_{d a}{}^d{}_b$ its Ricci
curvature and $S=g^{ab}R_{ab}$ its scalar curvature.

(i) The trace--free part $R_{ab}-\tfrac{S}{n+1}g_{ab}$ of the Ricci
tensor admits a smooth extension to the boundary.

(ii) Up to terms which admit a smooth extension to the boundary, the
curvature of $g_{ab}$ is given by
$$
R_{ab}{}^c{}_d=-\tfrac{1}{2\rho^2}\delta^c_{[a}\rho_{b]}\rho_d-
\tfrac{1}{2C\rho}\delta^c_{[a}h_{b]d}.
$$
\end{thm}
\begin{proof}
(i) By Proposition \ref{prop3.1} and formula \eqref{Rhoasymp} from its
proof, $\rho\Rho_{ab}+\frac1{4\rho}\rho_a\rho_b$ admits a smooth
extension to the boundary with boundary value
$\tfrac12\hat\nabla_a\rho_b$. On the other hand, Proposition
\ref{prop2.3} shows that $\rho\tfrac1{n+1}g^{ij}\Rho_{ij}
g_{ab}+\frac1{4\rho}\rho_a\rho_b$ admits a smooth extension to the
boundary. The boundary value of this coincides with the one of
$\tfrac1{n+1}g^{ij}\Rho_{ij}h_{ab}$ and hence with the one of
$-\frac1{4C}h_{ab}$. By Proposition \ref{prop3.2}, the latter boundary
value also equals $\tfrac12\hat\nabla_a\rho_b$. Forming the
difference, we conclude that
$\rho(\Rho_{ab}-\frac1{n+1}g^{ij}\Rho_{ij}g_{ab})$ admits
a smooth extension to the boundary with boundary value zero, so the
tracefree part of $\Rho_{ij}$ admits a smooth extension to the
boundary. Now in dimension $n+1$, we have $R_{ab}=\frac1n\Rho_{ab}$,
which implies the result.

\smallskip 

(ii) We use the formula for the curvature from the proof of
Proposition \ref{prop3.3}, taking into account that
$\be_{ab}=0$. Since we know from above, that
$\Rho_{ab}+\tfrac1{4C}g_{ab}$ admits a smooth extension to the
boundary, we may replace $\Rho_{ab}$ by $-\tfrac1{4C}g_{ab}$, and then
the claim follows from inserting the asymptotic form
$$
g_{ab}=\tfrac1{\rho}h_{ab}+\tfrac{C}{\rho^2}\rho_a\rho_b. 
$$ for $g$.
\end{proof}

\section{Boundary tractors}\label{4}
For the last part of this article, we assume that we have given a
special affine connection on $M$ which is projectively compact of
order two and has the property that the projective second fundamental
form is non--degenerate (in directions tangent to the boundary) at
each boundary point. (Observe that by Theorem \ref{thm2.5} and
Proposition \ref{prop3.2}, this condition is always satisfied in the
case of a pseudo--Riemannian metric which is projectively compact of
order two.) In this case, as shown is Section \ref{3.1}, a
well-defined conformal geometry is induced on the boundary $\partial
M$. As for any conformal geometry, its structure is naturally captured and
conceptually described by its associated conformal tractor bundle and
connection.

 In this section we give a description of these conformal boundary tractors in
 terms of the projective structure in the interior. We derive formulae
 for the ingredients used in this description both in terms of
 asymptotics of data associated to the projectively compact connection
 in the interior and in terms of data which are manifestly smooth up
 to the boundary. In contrast to the usual presentation of conformal
 tractors, our description is entirely based on connections from the
 projective class, we do not choose a connection on the boundary which
 is compatible with the conformal structure.

\subsection{The tractor bundle and its metric}\label{4.1}
In spite of the rather complicated relation between a projectively
compact connection on $M$ and the induced conformal structure on
$\partial M$, we show that the tractor
bundles associated to these structures are easily and elegantly related.  
As we have
observed in Section \ref{2.0}, a special affine connection $\nabla$ on
$M$, which is projectively compact of order two, determines a defining
density $\tau\in\Ga(\Cal E(2))$ for $\partial M$ (up to a non--zero
constant factor). The main property of $\tau$ is that, over $M$, it is
parallel for $\nabla$. Via the BGG splitting operator, we obtain a
section $L(\tau)$ of the tractor bundle $S^2\Cal T^*$ over $\barm$.

The motivation for the developments in this section comes from the
special case of Levi--Civita connections of non--Ricci--flat Einstein
metrics. In this case, the section $L(\tau)$ of $S^2\Cal T^*$ is
parallel for the the tractor connection, thus defining a reduction of
projective holonomy to a pseudo--orthogonal group. Via the general
theory of holonomy reductions developed in \cite{hol-red}, one obtains
an induced conformal structure on the boundary, which by Proposition
\ref{prop3.2} coincides with the one discussed in this article. The
general theory further implies that one can obtain the conformal
standard tractor bundle by restricting the projective standard tractor
bundle to the boundary, endowing it with the bundle metric
$L(\tau)$. Furthermore the restriction of the projective standard
tractor connection to this bundle is the conformal standard tractor
connection, see Sections 3.1 and 3.2 of \cite{hol-red}.

Surprisingly, the first part of this works in far greater generality,
as follows.

\begin{prop}\label{prop4.1} Let $\barm=M\cup\partial M$ be a smooth
  manifold with boundary, and suppose that $\nabla$ is a linear
  connection on $TM$ which is projectively compact of order two and
  such that the projective second fundamental form on $\partial M$ is
  non--degenerate.

Then endowing the restriction $\Cal T|_{\partial M}$ of the projective
standard tractor bundle with the line subbundle $\Cal T^1|_{\partial
  M}$ and the bundle metric $L(\tau)|_{\partial M}$, one obtains a
standard tractor bundle for the induced conformal structure on
$\partial M$.

Explicitly, this means that $\Cal T^1|_{\partial M}$ is isomorphic to
the conformal density bundle $\Cal E[-1]$ and isotropic for
$L(\tau)|_{\partial M}$, the quotient $(\Cal T^1)^\perp/\Cal T^1$ is
isomorphic to $T\partial M\otimes\Cal E[-1]$ and the metric on this
quotient induced by $L(\tau)$ coincides with the conformal metric
defined by the projective second fundamental form.
\end{prop}
\begin{proof}
In Section 3.3 of \cite{Proj-comp} it is shown that in the splitting
of $S^2\Cal T^*$ determined by $\nabla$ (which is only defined over
$M$) we have
\begin{equation}\label{Lform}
L(\tau)=\begin{pmatrix} \tau \\ 0\\ \Rho_{ab}\tau \end{pmatrix}. 
\end{equation}
Here we use that the Schouten tensor of a special affine connection is
symmetric. Now we can easily analyze the boundary behavior $L(\tau)$
analogously to the proof of Proposition \ref{prop2.3}. Consider a
local defining function $\rho$ for $\partial M$ and let
$\hat\nabla={}^\rho\nabla$ be the corresponding projectively rescaled
connection which admits a smooth extension to the boundary. Similar to
arguments in the proof of Proposition \ref{prop2.3}, we see that in the
splitting determined by $\hat\nabla$, we get
\begin{equation}\label{Lhatform}
L(\tau)=\begin{pmatrix} \rho\hat\tau
\\ \tfrac12\rho_a\hat\tau\\ \Rho_{ab}\rho\hat\tau+
\tfrac{\rho_a\rho_b}{4\rho}\hat\tau \end{pmatrix}. 
\end{equation}
Along the boundary, the top slot vanishes, while the middle slot is
evidently nowhere vanishing with pointwise kernel isomorphic to
$T\partial M\subset T\barm|_{\partial M}$. Finally, by formula
\eqref{Rhoasymp} from the proof of Proposition \ref{prop3.1}, the
boundary value of the bottom slot is
$\tfrac12\hat\tau\hat\nabla_a\rho_b$, so the restriction of this
bilinear form to boundary directions is non--degenerate by the
assumptions.

Together, this shows that $L(\tau)|_{\partial M}$ defines a
non--degenerate bundle metric on the restriction $\Cal T|_{\partial
  M}$ and that $\Cal T^1\subset\Cal T$ is isotropic for this bundle
metric along the boundary. Moreover, the form of the middle slot of
$L(\tau)$ in \eqref{Lhatform} implies that the quotient $(\Cal
T^1)^{\perp}/\Cal T^1$ can be identified with $T\partial M(-1)\subset
T\barm(-1)|_{\partial M}$.

Finally, recall that there is the canonical conormal bundle $\Cal
N\subset T^*\barm|_{\partial M}$, which is defined as the annihilator
of $T\partial M$. Now for the top exterior powers, we get
$(\La^{n+1}T^*\barm)|_{\partial M}\cong \Cal N\otimes
(\La^nT^*\partial M)$. In terms of the usual conventions for
projective and conformal density bundles (see \cite{BEG}) this reads
as $\Cal E(-n-2)|_{\partial M}\cong\Cal N\otimes\Cal E[-n]$. Now since
the top slot of $L(\tau)$ vanishes along $\partial M$, its middle slot
$\hat\tau\rho_a$ is actually independent of all choices, thus defining a
nowhere vanishing section of $\Cal N(2)\cong\Cal E(-n)|_{\partial
  M}\otimes\Cal E[n]$. In particular, this induces a canonical
isomorphism $\Cal E(n)|_{\partial M}\cong\Cal E[n]$ and hence also an
identification $\Cal E(-1)|_{\partial M}\cong\Cal E[-1]$.

This shows that we obtain the claimed composition series for $\Cal
T|_{\partial M}$. Since the bundle metric on $(\Cal T^1)^\perp/\Cal
T^1$ induced by $L(\tau)$ clearly comes from the restriction of
$\tfrac12\hat\tau\hat\nabla_a\rho_b$ to tangential directions, we also
get the correct conformal metric on the quotient.
\end{proof}

\subsection{The asymptotically parallel case}\label{4.1a} 
Without further assumptions, one can certainly not follow the
developments in the Einstein case discussed in \ref{4.1} directly,
since the projective standard tractor connection is not compatible
with the bundle metric $L(\tau)$. Indeed, the covariant derivative of
$L(\tau)$ with respect to the normal tractor connection on $S^2\Cal
T^*$ can be computed explicitly, see Section 3.3 of
\cite{Proj-comp}. There it is shown that, in the splitting on $M$
determined by the projectively compact connection $\nabla$, this
derivative is given by putting $\tau\nabla_a\Rho_{bc}$ into the bottom
slot of the tractor, while the other two slots are identical
zero. Since the bottom slot is the injecting slot, it has the same
form in any other splitting, so in particular, this section has to
admit a smooth extension to the boundary. We next give a direct proof
for the fact that $\tau\nabla_a\Rho_{bc}$ admits a smooth
extension. We also derive a formula for this tensor in terms of
objects which are manifestly smooth up to the boundary as well as an
alternative description, which is valid for Levi--Civita connections.

\begin{prop}\label{prop4.3}
Let $\nabla$ be a special affine connection on $M$, which is
projectively compact of order $2$ and induces a non--degenerate
boundary geometry on $\partial M$ and let $\Rho_{ab}$ be its Schouten
tensor. Let $\rho$ be a local defining function for the boundary and
let $\hat\nabla=\nabla+\frac{d\rho}{2\rho}$ be the corresponding
connection in the projective class. Then we have

(i) $ \rho\nabla_a\Rho_{bc}=\tfrac12\hat\nabla_a\hat\nabla_b\rho_c+
\rho_a\hat\Rho_{bc}+\tfrac12\rho_b\hat\Rho_{ac}+\tfrac12\rho_c\hat\Rho_{ab}+
\rho\hat\nabla_a\hat\Rho_{bc}$, and the right hand side provides a
smooth extension of the left hand side to the boundary.

(ii) If $\nabla$ is the Levi--Civita connection of a
pseudo--Riemannian metric $g_{ab}$, and $S$ is its scalar curvature,
then for $\Ph_{ab}:=\Rho_{ab}-\tfrac1{n(n+1)}Sg_{ab}$, we get
$$
\rho\nabla_a\Rho_{bc}=\rho_a\Ph_{bc}+\tfrac12\rho_b\Ph_{ac}+
\tfrac12\rho_c\Ph_{ba}+\rho\left(\hat\nabla_a\Ph_{bc}+
\tfrac1{n(n+1)}g_{bc}\hat\nabla_aS\right).  
$$ 
All terms in the right hand side admit smooth extensions to the
boundary and the last summand does not contribute to the boundary
value.
\end{prop}
\begin{proof}
(i) For $\al=2$, equation \eqref{Rhoasymp} from the proof of
  Proposition \ref{prop3.1} reads as
\begin{equation}\label{Rhoasymp2}
\rho\Rho_{bc}+\tfrac1{4\rho}\rho_b\rho_c=\tfrac12\hat\nabla_b\rho_c+
\rho\hat\Rho_{bc}.  
\end{equation}
Applying $\hat\nabla_a$ to this equation, the second term on the left
hand side gives
\begin{equation}\label{techrho}
\tfrac{-1}{4\rho^2}\rho_a\rho_b\rho_c+\tfrac1{4\rho}\rho_b\hat\nabla_a\rho_c+
\tfrac1{4\rho}\rho_c\hat\nabla_a\rho_b .
\end{equation}
Now we can combine half of the first summand in this expression with
the second summand to obtain
$$
\tfrac{1}{4\rho}\rho_b(\hat\nabla_a\rho_c-\tfrac1{2\rho}\rho_a\rho_c). 
$$ 
From \eqref{Rhoasymp2} we see that we can replace the bracket by
$2\rho(\Rho_{ac}-\hat\Rho_{ac})$ and thus obtain 
$$
\tfrac12\rho_b\Rho_{ac}-\tfrac12\rho_b\hat\Rho_{ac}. 
$$ 
Likewise the second half of the first term in \eqref{techrho} adds up
with the last term in this formula to the same expression with $b$ and $c$
exchanged. 

To compute $\hat\nabla_a$ of the first term in the left hand side of
\eqref{Rhoasymp2} we use the standard formulae for the action of
projectively related connections on tensor fields to obtain
$$
\hat\nabla_a\Rho_{bc}=\nabla_a\Rho_{bc}-2\Ups_a\Rho_{bc}-\Ups_b\Rho_{ac}-
\Ups_c\Rho_{ba}.
$$
Here $\Ups$ describes the change from $\nabla$ to $\hat\nabla$,
i.e.~$\Ups_a=\tfrac{\rho_a}{2\rho}$. We have to multiply all that by
$\rho$ and add $\rho_a\Rho_{bc}$ to obtain the contribution of the
first term on the left hand side. Hence we conclude that applying
$\hat\nabla_a$ to the left hand side of \eqref{Rhoasymp2} we obtain 
$$ 
\rho\nabla_a\Rho_{bc}-\tfrac12\rho_b\hat\Rho_{ac}-\tfrac12\rho_c\hat\Rho_{ab}.
$$
Applying $\hat\nabla_a$ to the right hand side of \eqref{Rhoasymp2}
directly leads to the claimed formula. 

(ii) Observe first that $\Ph_{ab}$ admits a smooth extension to the
boundary by Theorem \ref{thm3.3}. Since $S$ admits a smooth extension
to the boundary by Proposition \ref{prop2.1}, the last statement is
evident. On $M$, we obtain
$$
\hat\nabla_a\Ph_{bc}=\nabla_a\Ph_{bc}-2\Ups_a\Ph_{bc}-\Ups_b\Ph_{ac}-
\Ups_c\Ph_{ba},
$$ 
as in the proof of part (i) with $\Ups_a=\frac{\rho_a}{2\rho}$. Now
$\nabla_a\Ph_{bc}=\nabla_a\Rho_{bc}-\tfrac{1}{n(n+1)}g_{bc}\nabla_aS$,
and since $S$ is a function, can replace $\nabla_a$ by $\hat\nabla_a$
in the last term. From this the claimed formula follows immediately by
multiplying by $\rho$ and rearranging terms.
\end{proof}

As mentioned in \ref{4.1}, in the case of the Levi--Civita connection
of an Einstein metric, the bundle metric $L(\tau)$ is parallel over
all of $\barm$, and one obtains the conformal standard tractor
connection on the boundary as a restriction of the projective standard
tractor connection. The argument which was used to prove this in
Proposition 3.2 of \cite{hol-red} actually can be applied in a
significantly more general situation, as we will show next. 

Surprisingly, it suffices to assume that $\nabla^{S^2\Cal T^*}L(\tau)$
vanishes along the boundary (although this is not enough to ensure
compatibility of the tractor curvature with $L(\tau)$ along the
boundary). Since $\nabla^{S^2\Cal T^*}L(\tau)$ amounts to
$\tau\nabla_a\Rho_{bc}$, in the sense described above, and by
Proposition \ref{prop4.3} (for example) this has a smooth extension to
the boundary, it follows that $\nabla^{S^2\Cal T^*}L(\tau)$ vanishes
on $\partial M$ if and only if $\nabla_a\Rho_{bc}$ admits a smooth
extension to all of $\barm$. Moreover, from Proposition \ref{prop4.3}
we see that, for a pseudo--Riemannian metric $g_{ab}$ which is
projectively compact of order two, vanishing of $\nabla^{S^2\Cal
  T^*}L(\tau)$ along $\partial M$ is equivalent to the
boundary value of $R_{ab}-\tfrac{S}{n+1}g_{ab}$ vanishing
identically. The last condition is a (by one order) stronger
asymptotic form of the Einstein equation than the one that $g_{ab}$
satisfies by Theorem \ref{thm3.3}.

\begin{thm}\label{thm4.1a}
 Let $\barm=M\cup\partial M$ be a smooth manifold of dimension
 $n+1\geq 4$ with boundary and suppose that $\nabla$ is a linear
 connection on $TM$ which is projectively compact of order two and
 such that the projective second fundamental form on $\partial M$ is
 non--degenerate. Assume further that the canonical defining density
 $\tau\in\Ga(\Cal E(2))$ for $\partial M$ determined by $\nabla$ has
 the property that $\nabla^{S^2\Cal T^*}L(\tau)|_{\partial M}=0$.

Then one can restrict the projective standard tractor connection to
the conformal standard tractor bundle on $\partial M$, as constructed in
Proposition \ref{prop4.1}, and the result is the canonical normal
conformal tractor connection.
\end{thm}
\begin{proof}
It is no problem to restrict the tractor connection on $\Cal
T\to\barm$ to a linear connection on $\Cal T|_{\partial M}\to\partial
M$. Since we have assumed that $\nabla^{S^2\Cal T^*}L(\tau)|_{\partial
  M}=0$, this produces a tractor connection, which is compatible with
the bundle metric $L(\tau)|_{\partial M}$. To complete the proof, it
remains to verify that the curvature of this tractor connection
satisfies the normalization condition imposed on a conformal standard
tractor connection. 

This normalization condition is best described in two steps. The first
requirement on the curvature is that it maps the distinguished
subbundle $\Cal T^1|_{\partial M}$ to itself. Skew symmetry of the
curvature then implies that is also preserves the orthocomplement of
this subbundle, so there is an induced endomorphism on the quotient
space, which is isomorphic to $T\partial M\otimes\Cal E(-1)$. One can
view the result as a section of $\La^2T^*\partial
M\otimes\End(T\partial M)$, and the second part of the normalization
condition is that the Ricci--type contraction of this tensor field
vanishes.

Now the curvature of the restricted connection is just the restriction
of the curvature of the projective standard tractor connection. This
means that one only inserts vectors tangent to the boundary into the
two--form part of the curvature, but the endomorphism part still acts
on the full bundle. It is well known (see \cite{BEG}) that the
curvature of the projective standard tractor connection satisfies
similar normalization conditions. In particular, this curvature
vanishes identically on the distinguished subbundle $\Cal
T^1$. Similarly as above, this implies that the values of the
curvature descend to endomorphisms of the quotient $\Cal T/\Cal
T^1\cong T\barm(-1)$. So one obtains a section of
$\La^2T^*\barm\otimes\End(T\barm)$ and the Ricci--type contraction of
this vanishes (and the tensor itself coincides with the projective
Weyl curvature of any connection in the projective class).

Now the fact that the subbundle $\Cal T^1$ is annihilated of course
carries over to the restriction, so the first part of the conformal
normalization condition is satisfied. Now suppose that we can further
show that values of the endomorphisms obtained from the projective
Weyl curvature along the boundary always lie in $T\partial M\subset
T\barm|_{\partial M}$. Then using a basis of $T_x\barm$ consisting of
a basis of $T_x\partial M$ for $x\in\partial M$ and one transversal
vector, one immediately concludes that the Ricci type contraction of
the projective Weyl curvature coincides with the Ricci--type
contraction over the subspaces $T\partial M$, so that latter
vanishes. Hence we can complete the proof by verifying this property
of the projective Weyl curvature. This can be done by taking a locally
non--vanishing section $\si$ of $\Cal T^1$ and proving that, denoting
by $\ka$ the curvature of the projective tractor connection
$\nabla^{\Cal T}$, we get
$$
L(\tau)(\ka(\xi,\eta)(t),\si)|_{\partial M}=0
$$ for all $\xi,\eta\in\frak X(\barm)$ and any section $t\in\Ga(\Cal
T)$. (We could actually assume in addition that $\xi$ is tangent to
$\partial M$ and that $L(\tau)(t,\si)=0$, but these assumptions are
not needed.) Note that this would follow immediately if we were to
assume that the one--jet of $\nabla^{S^2\Cal T^*}L(\tau)$ vanishes
along $\partial M$, since this implies skew symmetry of
$\ka(\xi,\eta)$ with respect to $L(\tau)$ along the boundary.

Under the weaker assumptions we have made, we have to supply a direct
argument which uses the additional information on $\nabla^{S^2\Cal
  T^*}L(\tau)$ we have available. We start with the defining equation
$$ (\nabla^{S^2\Cal T^*}_\xi
L(\tau))(t_1,t_2)=\xi\cdot(L(\tau)(t_1,t_2))-L(\tau)(\nabla^{\Cal
  T}_\xi t_1,t_2)-L(\tau)(t_1,\nabla^{\Cal T}_\xi t_2)
$$
for $t_1,t_2\in\Ga(\Cal T)$. Using this, one directly computes that
\begin{align*}
L(\tau)(\nabla^{\Cal T}_\xi&\nabla^{\Cal T}_\eta
t_1,t_2)-L(\tau)(t_1,\nabla^{\Cal T}_\eta\nabla^{\Cal T}_\xi t_2)\\
=& \xi\cdot(L(\tau)(\nabla^{\Cal T}_\eta t_1,t_2))-(\nabla^{S^2\Cal
  T^*}_\xi L(\tau))(\nabla^{\Cal T}_\eta t_1,t_2)\\
-&\eta\cdot(L(\tau)(t_1,\nabla^{\Cal T}_\xi
t_2))+(\nabla^{S^2\Cal T^*}_\eta L(\tau))(t_1,\nabla^{\Cal T}_\xi t_2).
\end{align*}
Observe that the terms involving a covariant derivative of $L(\tau)$
by assumption vanish along the boundary, so we can drop them for the
further considerations. Subtracting the same term with $\xi$ and
$\eta$ exchanged, for the right hand side we obtain 
\begin{align*}
\xi\cdot \big( L(\tau)(\nabla^{\Cal T}_\eta t_1,t_2)+&L(\tau)(t_1,\nabla^{\Cal T}_\eta t_2) \big)\\
=&\xi\cdot\eta\cdot (L(\tau)(t_1,t_2))-\xi\cdot\left(\nabla^{S^2\Cal
  T^*}_\eta L(\tau)(t_1,t_2)\right), 
\end{align*}
minus the same expression with $\xi$ and $\eta$ exchanged. Now the
second term in the right hand side here does not vanish along the boundary
in general. However, we only have to consider this in the case that
$t_2=\si\in\Ga(\Cal T^1)$. But the fact that $\nabla^{S^2\Cal T^*}
L(\tau)$ is concentrated in the bottom slot (over all of $\barm$)
which we have noted in the beginning of Section \ref{4.1a} exactly
means that any covariant derivative of $L(\tau)$ vanishes identically
provided that one of its entries is from the subbundle $\Cal T^1$. So
the only potential contribution to the boundary value coming from
these two terms is
\begin{equation}\label{lasttech}
\xi\cdot\eta\cdot (L(\tau)(t_1,t_2))-\eta\cdot\xi\cdot (L(\tau)(t_1,t_2)).
\end{equation}

To arrive at
$$
L(\tau)(\ka(\xi,\eta)(t_1),t_2)+L(\tau)(t_1,\ka(\xi,\eta)(t_2)),
$$
we further have to subtract 
\begin{align*}
L(\tau)&(\nabla^{\Cal T}_{[\xi,\eta]}t_1,t_2)+L(\tau)(t_1,\nabla^{\Cal
  T}_{[\xi,\eta]}t_2)\\
&=[\xi,\eta]\cdot (L(\tau)(t_1,t_2))-(\nabla^{S^2\Cal
  T^*}_{[\xi,\eta]} L(\tau))(t_1,t_2). 
\end{align*}
Now the first term on the right hand side cancels with
\eqref{lasttech}, while the second one vanishes along the boundary by
assumption. Now the claim follows since $\ka(\xi,\eta)$ vanishes on
the subbundle $\Cal T^1$. 
\end{proof}

\subsection{The inverse of the tractor metric}\label{4.2} 
Before we can proceed towards the description of the normal tractor
connection on the boundary in the case that $L(\tau)$ is not parallel
along the boundary, we have to derive some further properties of the
Schouten--tensor $\Rho_{ab}$ of $\nabla$. In Proposition \ref{prop4.1}
we have seen that non--degeneracy of the boundary geometry implies
that the bundle metric $L(\tau)$ is non--degenerate on $\partial
M$. By continuity, it is non--degenerate on some open neighborhood of
the boundary and we will from now on restrict to this neighborhood,
i.e.~assume that $L(\tau)$ is non--degenerate on all of $\barm$. On
$M$ we can return to the scale determined by $\tau$, and there, in
view of \eqref{Lform}, non--degeneracy of $L(\tau)$ is equivalent to
non--degeneracy of the Schouten--tensor $\Rho_{ab}$. This means that
we can use $\Rho_{ab}$ as a Riemannian metric on $M$, but of course,
the Levi--Civita connection of this metric is not in the projective
class in general.

By non--degeneracy, we can also form the inverse $\Rho^{ab}$ of
$\Rho_{ab}$ as a bilinear form. We can derive asymptotic properties of
$\Rho^{ab}$ using the inverse $L(\tau)^{-1}$ of the tractor metric,
which is a smooth section of $S^2\Cal T$ over all of $\barm$.

\begin{prop}\label{prop4.2}
Let $\nabla$ be a special affine connection on $M$, which is
projectively compact of order $2$ and induces a non--degenerate
boundary geometry on $\partial M$. Let $\rho$ be a local defining
function for the boundary and let
$\hat\nabla=\nabla+\frac{d\rho}{2\rho}$ be the corresponding
connection in the projective class. Then in the splitting of $S^2\Cal
T$ defined by the connection $\hat\nabla$, the inverse $L(\tau)^{-1}$
of the tractor metric is given by
\begin{equation}\label{L-1hatform}
L(\tau)^{-1}=\begin{pmatrix} \hat\tau^{-1}\rho^{-1}\Rho^{ab}
\\ 2\hat\tau^{-1} t^a\\\hat\tau^{-1} \ps \end{pmatrix}
\end{equation}
where $\tau=\rho\hat\tau$, $t^a=-\tfrac1{4\rho^2}\Rho^{ab}\rho_b$, and
$\ps$ is a function which is smooth up to the boundary. Moreover, we
obtain
\begin{equation}\label{split-ids}
\begin{aligned} &t^a\rho_a=1-\rho\ps \qquad
  t^a(\rho\Rho_{ab}+\tfrac{1}{4\rho}\rho_a\rho_{b})=-\tfrac14\ps\rho_b
  \\ &\rho^{-1}\Rho^{ac}(\rho\Rho_{cb}+\tfrac{1}{4\rho}\rho_c\rho_b)+
 t^a\rho_b=\delta^a_b
\end{aligned}
\end{equation}

In particular, the tensor fields $\rho^{-1}\Rho^{ab}$ and
$\rho^{-2}\Rho^{ab}\rho_b$ on $M$ admit smooth extensions to all of
$\barm$.
\end{prop}
\begin{proof}
Over $M$, and in the splitting determined by $\nabla$, we clearly
have
\begin{equation}\label{L-1form}
L(\tau)^{-1}=\begin{pmatrix} \tau^{-1}\Rho^{ab}\\ 0 \\ \tau^{-1}
\end{pmatrix}. 
\end{equation}
The top slot of this is independent of the choice of splitting, so we
see that we can use \eqref{L-1hatform} to define $t^a$ and $\ps$. But
then we can use formula \eqref{Lhatform} for $L(\tau)$ in the
splitting determined by $\hat\nabla$ to compute the consequences of
$L(\tau)$ and $L(\tau)^{-1}$ being inverses of each other. This is
most easily done by interpreting $L(\tau)$ as a map $\Cal T\to\Cal
T^*$ and $L(\tau)^{-1}$ as a map $\Cal T^*\to\Cal T$. By
\eqref{Lhatform} in the splitting determined by $\hat\nabla$, we have
$$
L(\tau)\left(\binom{\nu_1^a}{\si_1},\binom{\nu_2^b}{\si_2}\right)=
\hat\tau\left(\rho\si_1\si_2+\tfrac12\si_1\rho_a\nu_2^a+
\tfrac12\si_2\rho_a\nu_1^a+(\rho\Rho_{ab}+
\tfrac1{4\rho}\rho_a\rho_b)\nu_1^a\nu_2^b\right). 
$$
Hence the associated map is given by 
$$ 
\binom{\nu^a}{\si}\mapsto
\binom{\hat\tau(\rho\si+\tfrac12\rho_a\nu^a)}
      {\hat\tau(\tfrac12\si\rho_a+(\rho\Rho_{ab}+
        \tfrac1{4\rho}\rho_a\rho_b)\nu^b)}
$$
In the same way, one verifies that \eqref{L-1hatform} corresponds to
the map $\Cal T^*\to\Cal T$ given by
$$
\binom{\be}{\mu_a}\mapsto\binom{\hat\tau^{-1}(\rho^{-1}\Rho^{ab}\mu_b+2\be
  t^a)}{\hat\tau^{-1}(2t^a\mu_a+\ps\be)}.
$$ The fact that the composition of this with the above is the
identity immediately leads to the claimed formula for $t^a$ as well as
to \eqref{split-ids}. The last claim then follows since all slots in
\eqref{L-1hatform} must admit smooth extensions to the boundary.
\end{proof}

\subsection{The metric tractor connection}\label{4.3} 
Now we can proceed, in the general setting, toward a description of
the normal tractor connection on the conformal standard tractor bundle
obtained in Proposition \ref{prop4.1}. We will
do this in two steps, the first of which can be done on all of $\barm$
(assuming that $L(\tau)$ is non--degenerate on all of $\barm$). In
this first step, we modify the projective standard tractor connection
on $\Cal T$ to a connection which is compatible with the bundle metric
$L(\tau)$ and torsion free (in the sense of tractor connections). In
the second step, we have to restrict to the boundary, where we can
then normalize this metric tractor connection to obtain the conformal
standard tractor connection.

A modification of the standard tractor connection $\nabla^{\Cal T}$ is
determined by a contorsion, which is an element of
$\Om^1(\barm,\End(\Cal T))$. Choosing a connection in the projective
class, one obtains an isomorphism $\Cal T\cong \Cal E(-1)\oplus\Cal
E^a(-1)$ and correspondingly we get an isomorphism $\End(\Cal T)\cong
\Cal E_b\oplus (\Cal E^a_b\oplus\Cal E(0))\oplus\Cal E^a$. We write
this in a matrix form, with the action given by
$$
\begin{pmatrix} A^a{}_b & \xi^a \\ \ps_b & \la 
\end{pmatrix}\begin{pmatrix} \nu^b \\ \si \end{pmatrix}=
\begin{pmatrix} A^a{}_b\nu^b+\si\xi^a
  \\ \la\si+\ps_b\nu^b \end{pmatrix}. 
$$
From this definition and the change of splitting on standard tractors
as described in \cite{BEG}, one readily concludes that a change of
connection described by a one--form $\Ups_a$ changes this splitting as
\begin{equation}\label{End-split}
\begin{aligned}
&\hat\xi^a=\xi^a \qquad \hat A^a{}_b=A^a{}_b+\xi^a\Ups_b \qquad
  \hat\la=\la-\Ups_c\ph^c \\
&\hat\ps_b=\ps_b-A^c{}_b\Ups_c+\la\Ups_b-\Ups_c\xi^c\Ups_b . 
\end{aligned}
\end{equation}
Analogously, we can denote one--forms with values in $\End(\Cal T)$
by simply adding an additional lower index to each slot. It is also
straightforward to describe the linear connection on $\End(\Cal T)$
induced be the standard tractor connection. In terms of any connection
$\tilde\nabla_a$ in a projective class with Schouten--tensor
$\tilde\Rho_{ab}$ the standard tractor connection is, in the splitting
determined by $\tilde\nabla_a$, given by 
\begin{equation}\label{std-conn}
\nabla^{\Cal T}_a\binom{\nu^b}{\si}=
\binom{\tilde\nabla_a\nu^b+\si\delta_a^b}
{\tilde\nabla_a\si-\tilde\Rho_{ab}\nu^b},
\end{equation}
see \cite{BEG}. From this, one deduces by a straightforward
computation that the induced linear connection on $\End(\Cal T)$ is,
in that splitting, given by
\begin{equation}\label{End-conn}
\nabla^{\End(\Cal T)}_a\begin{pmatrix} A^b{}_c & \xi^b \\ \ps_c & \la 
\end{pmatrix}=\begin{pmatrix} \tilde\nabla_a A^b{}_c +
  \ps_c\delta_a^b+\tilde\Rho_{ac}\xi^b & \tilde\nabla_a
  \xi^b+\la\delta_a^b-A^b{}_a  \\ \tilde\nabla_a\ps_c
  -\tilde\Rho_{ad}A^d{}_c -\la\tilde\Rho_{ac} & \tilde\nabla_a\la
  -\tilde\Rho_{ad}\xi^d-\ps_a \end{pmatrix}
\end{equation}

Now we can compute the torsion free metric connection and its
curvature.

\begin{thm}\label{thm4.3}
Given $\nabla_a$ as before, consider the $\End(\Cal T)$--valued
one--form $\Ps$, which on $M$ is defined in the splitting corresponding
to $\nabla_a$ as having all entries equal to zero, except for
$$
A_a{}^b{}_c:=\tfrac12\Rho^{bd}(-\nabla_a\Rho_{dc}-\nabla_c\Rho_{da}+
\nabla_d\Rho_{ac}).
$$ 
Then we have:

(i) $\Ps$ admits a smooth extension to all of $\barm$ and, defining a
modification of the tractor connection as $\tilde\nabla^{\Cal T}_\xi
s:=\nabla^{\Cal T}_\xi s+\Ps(\xi)(s)$, the resulting connection is
metric for $L(\tau)$.

(ii) Consider a local defining function $\rho$ for $\partial M$, let
$\hat\nabla_a$ be the corresponding connection in the projective
class, $C_{ab}{}^c{}_d$ its projective Weyl curvature and $Y_{abc}$
its projective Cotton tensor. Further, let $t^a$ be the vector field
from Proposition \ref{prop4.2}, and put 
$$
\ps_{ac}:=t^d\rho(-\nabla_a\Rho_{dc}-\nabla_c\Rho_{da}+
\nabla_d\Rho_{ac}). 
$$ 
Then as a two--form with values in $\End(\Cal T)$, the curvature of
$\tilde\nabla^{\Cal T}$ is, in the splitting determined by
$\hat\nabla$, given by
$$
\begin{pmatrix}
C_{ab}{}^c{}_d+2\hat\nabla_{[a}A_{b]}{}^c{}_d-2\ps_{d[a}\delta^c_{b]}+
2A_e{}^c{}_{[a}A_{b]}{}^e{}_d  & 0 \\
Y_{abd}+2\hat\nabla_{[a}\ps_{b]d}-2\hat\Rho_{e[a}A_{b]}{}^e{}_d+2\ps_{e[a}A_{b]}{}^e{}_d
& 0
\end{pmatrix},
$$ so in particular, $\tilde\nabla^{\Cal T}$ is a torsion free tractor
connection.
\end{thm}
\begin{proof}
Take a local defining function $\rho$ for $\partial M$ and let
$\hat\nabla$ be the corresponding connection in the projective class
which is smooth up to the boundary. Then from \eqref{End-split} we see
that writing $\Ps$ over $M$ in the splitting corresponding to
$\hat\nabla$, there are two non--zero entries, namely $\hat
A_a{}^b{}_c=A_a{}^b{}_c$ and $\hat\ps_{ac}=-A_a{}^d{}_c\Ups_d$, and we
will omit the hats in the notation for $A$ and $\ps$ from now on. From
Propositions \ref{prop4.2} and \ref{prop4.3}, we know that
$\rho^{-1}\Rho^{ab}$ and $\rho\nabla_a\Rho_{bc}$ admit smooth
extensions to all of $\barm$, whence the same is true for
$A_a{}^b{}_c$. On the other hand $\Ups_a=\tfrac{\rho_a}{2\rho}$, so
again by Proposition \ref{prop4.2}, $\ps_{ac}$ admits a smooth
extension to the boundary and has the claimed form. 

Knowing that $\tilde\nabla^{\Cal T}$ is well defined on all of
$\barm$, it suffices to prove that it is metric on the dense open
subset $M$, where we can compute in the scale determined by
$\nabla$. In that scale, formula \eqref{Lform} for $L(\tau)$ shows
that
$$
L(\tau)\left(\binom{\nu_1^a}{\si_1},\binom{\nu_2^b}{\si_2}\right)=
\tau\si_1\si_2+\tau\Rho_{ab}\nu_1^a\nu_2^b. 
$$
On the other hand, 
$$
\tau\Rho_{bc}A_a{}^b{}_d\nu_1^d\nu_2^c=\tfrac12\tau\nu_1^d\nu_2^c
(-\nabla_a\Rho_{cd}-\nabla_d\Rho_{ca}+\nabla_c\Rho_{ad}), 
$$ and adding the same term with $\nu_1$ and $\nu_2$ exchanged, we
arrive at $-\nu_1^d\nu_2^c\tau\nabla_a\Rho_{cd}$. Using this, and
formula \eqref{std-conn} for the standard tractor connection it is
easy to verify by a direct computation that $\tilde\nabla^{\Cal T}_a$
is metric for $L(\tau)$.

For the description of the curvature, we use the description of $\Ps$
in the splitting corresponding to $\hat\nabla$ from above. The
curvature of the standard tractor connection is well known to be given
in a splitting by the Weyl--curvature and the Cotton tensor, see
\cite{BEG}. On the other hand, it is also well known that the
definition of $\tilde\nabla^{\Cal T}$ implies that its curvature is
related to the one of $\nabla^{\Cal T}$ by 
$$
\tilde R(\xi,\eta)=R(\xi,\eta)+\nabla^{\End(\Cal
  T)}_\xi(\Ps(\eta))-\nabla^{\End(\Cal
  T)}_\eta(\Ps(\xi))-\Ps([\xi,\eta])+[\Ps(\xi),\Ps(\eta)], 
$$ where in the last term we use the commutator of endomorphisms. The
second to fourth term in the right hand side are the covariant exterior
derivative of the $\End(\Cal T)$--valued one--form $\Ps$ with respect
to the connection induced by $\nabla^{\Cal T}$. This can be computed
by coupling $\nabla^{\Cal T}$ to the (torsion free) connection
$\hat\nabla$ on $T^*M$, differentiating the one--form $\Ps$ with this
coupled connection and then taking the alternation in the form--indices
and multiplying by two. Using all that, the fact that both $A$ and $\ps$
are symmetric in the lower indices, and formula \eqref{End-conn} for
$\nabla^{\End(\Cal T)}$, the claimed formula for the curvature follows
by a direction computation, and torsion freeness just means that the
top right entry in the resulting matrix vanishes.
\end{proof}

Observe that inserting the descriptions of $\rho\nabla_a\Rho_{bc}$
from Proposition \ref{prop4.3} into the formulae for $A_a{}^b{}_c$
and $\ps_{ac}$ from the theorem, there are some cancellations. For
example, in the case of a Levi--Civita connection, we obtain
$$
A_a{}^b{}_c|_{\partial
  M}=-\tfrac12\rho^{-1}\Rho^{bd}(\rho_a\Ph_{dc}+\rho_c\Ph_{da})|_{\partial
  M}, 
$$ where $\Ph_{ab}$ is the tensor from Proposition \ref{prop4.3}. A
similar expression holds for $\ps_{ac}$.

\subsection{Restricting to the boundary}\label{4.4} 
Over $M$, the connection $\tilde\nabla^{\Cal T}$ constructed in
Theorem \ref{4.3} is essentially uniquely determined by compatibility
with the bundle metric $L(\tau)$ and torsion freeness. (This is
closely related to the proof of existence and uniqueness of the
Levi--Civita connection in the Cartan picture. Likewise, the proof of
Theorem \ref{4.3} is closely related to the construction of the
Levi--Civita connection.) However, if we restrict to the boundary and
differentiate only in boundary directions a further normalization is
possible, and this will lead to a description of the conformal
standard tractor connection. We do not provide complete formulae in
general, but only describe how they can be obtained. The problem is
that formulae are getting quite involved without simplifying
assumptions (which is not surprising in view of the rather complicated
relation between the geometries in the interior and on the boundary).

In what follows, we have to distinguish between directions tangent to
the boundary and transversal directions, and we will adapt the
abstract index notation accordingly.  We use indices $i$, $j$, $k$,
and so on to specify boundary directions, while indices $a$, $b$, $c$,
and so on will be used for directions which are not necessarily
tangent to the boundary. A certain amount of care is needed here and
also upper and lower indices have to be distinguished. For a lower
index, it is no problem to replace a ``general'' index by a
``tangential'' one; this simply corresponds to restricting a linear
functional to a hyperplane. On the other hand, there is no canonical
extension of a functional defined on a hyperplane to the whole space,
so ``tangential'' lower indices cannot be replaced by ``general'' ones
without further choices. In contrast to this, for upper indices, a
``tangential'' index can always be considered as a general one
(corresponding to the inclusion of a hyperplane into a vector
space). One can recognize tangential upper indices by the fact that
they have trivial contraction with $\rho_a$. 

From now on, let us fix a local defining function $\rho$ for the
boundary and the corresponding connection $\hat\nabla_a$ in the
projective class. Then consider the quantity
$\ga_{ab}:=\rho\Rho_{ab}+\frac1{4\rho}\rho_a\rho_b$ which occurs in
\eqref{Lhatform}. This admits a smooth extension to the boundary, and
indeed by formula \eqref{Rhoasymp2} from the proof of Proposition
\ref{prop4.3}, we get $\ga_{ab}=\tfrac12\hat\nabla_a\rho_b+
\rho\hat\Rho_{ab}$. In particular, restricting to the boundary and
tangential directions, we can form $\ga_{ij}$, and this is a
representative of the projective second fundamental form. On the other
hand, consider the quantity $\rho^{-1}\Rho^{ab}$ which shows up in
\eqref{L-1hatform}. From the fact that the vector field $t^a$ showing
up in this proposition is smooth up to the boundary, we see that
$\rho^{-1}\Rho^{ab}\rho_b$ vanishes along $\partial M$. Hence the
restriction of $\rho^{-1}\Rho^{ab}$ to $\partial M$ is actually
tangential. Then the last equation in \eqref{split-ids} shows that on
tangential vectors, this restriction is actually inverse to
$\ga_{ij}$, so we denote it by $\ga^{ij}$.

Next, we introduce a finer decomposition of $\Cal T|_{\partial M}$,
which resembles the usual picture of conformal standard tractors in
slots. The necessary information is basically contained in Proposition
\ref{prop4.2}. In particular, we can use the transversal $t^a$ from
there to identify $T\bar M(-1)$ along $\partial M$ with $\Cal
E(1)\oplus T\partial M(-1)$ according to
\begin{equation}\label{split-def}
\nu^a\mapsto \binom{\hat\tau \nu^b\rho_b}{\nu^a-\nu^b\rho_bt^a}\qquad
\binom{\be}{\xi^i}\mapsto \xi^a+\hat\tau^{-1}\be t^a.   
\end{equation}
These are inverse to each other since by the first formula in
\eqref{split-ids}, we have $t^b\rho_b=1$ along $\partial M$.  Now we
combine this with the splitting of $\Cal T$ determined by
$\hat\nabla_a$ to identify $\Cal T$ along $\partial M$ with $\Cal
E(1)\oplus T\partial M(-1)\oplus\Cal E(-1)$, and we will use this
splitting from now on. The second formula in \eqref{split-ids} says
that $t^a\ga_{ab}=-\tfrac14\ps\rho_b$. Combining this with the formula
for $L(\tau)$ from the proof of Proposition \ref{prop4.2}, we get 
\begin{equation}\label{tract-met-split}
L(\tau)\left(\left(\begin{smallmatrix} \be_1 \\ \xi_1^i \\ \si_1
\end{smallmatrix}\right)\left(\begin{smallmatrix} \be_2 \\ \xi_2^i \\ \si_2
\end{smallmatrix}\right)\right)=\tfrac12\be_1\si_2+\tfrac12\be_2\si_1+
\hat\tau\ga_{ij}\xi_1^i\xi_2^j-\tfrac14\ps\hat\tau^{-1}\be_1\be_2. 
\end{equation}
This is slightly different from the usual splitting of conformal
standard tractors, since the line spanned by the first basis vector
(corresponding to the $\Cal E(1)$--component) is not isotropic. One
could correct this by passing to a line in $\Cal T$ which is not
contained in the subspace $\Cal E^a(-1)$ (in the splitting determined
by $\hat\nabla$), but we do not do this at this stage. 

\begin{thm}\label{thm4.4}
Consider the restriction of the linear connection $\tilde\nabla^{\Cal
  T}$ from Theorem \ref{thm4.3} to the boundary (i.e.~we differentiate
in boundary directions only). Then the curvature of this restriction
is given by restricting the two--form indices $a$ and $b$ in the
formula for the curvature in Theorem \ref{thm4.3} to tangential
directions. Moreover, in our splitting, the curvature takes the form 
$$
\begin{pmatrix}  0 & 0 & 0 \\ V_{ij}{}^k & W_{ij}{}^k{}_\ell & 0 \\ 
0 & -2\hat\tau V_{ij}{}^k\ga_{k\ell} & 0\end{pmatrix} 
$$
where $V_{ij}{}^k=V_{[ij]}{}^k$,
$W_{ij}{}^k{}_\ell=W_{[ij]}{}^k{}_\ell$ and
$W_{ij}{}^r{}_\ell\ga_{kr}=-W_{ij}{}^r{}_k\ga_{\ell r}$.  

Suppose further that $n=\dim(\partial M)\geq 3$. Then putting 
$$
\ph_{ij}:=-\tfrac1{n-2}W_{ki}{}^k{}_j+
\tfrac1{2(n-1)(n-2)}W_{kr}{}^k{}_s\ga^{rs}\ga_{ij}
$$
and defining an $\End(\Cal T)$--valued one--form $\tilde\Ps$ on $\partial M$
in our splitting by 
$$
\begin{pmatrix}  0 & 0 & 0
  \\ -\tfrac12\hat\tau^{-1}\ph_{i\ell}\ga^{k\ell} & 0 & 0 \\  
0 & \ph_{ij} & 0\end{pmatrix},
$$ the linear connection $\nabla^0$ defined by $\nabla^0_\xi
s=\tilde\nabla^{\Cal T}_\xi s+\tilde\Ps(\xi)(s)$ is the normal
conformal tractor connection on the conformal tractor bundle $\Cal
T|_{\partial M}\to\partial M$.
\end{thm}
\begin{proof}
The facts that $\tilde\nabla^{\Cal T}$ can be restricted to the
boundary and that the curvature is obtained by restriction follows as
in the proof of Theorem \ref{thm4.1a}. Writing the resulting curvature
in a matrix according to our splitting, we see from the formula in
Theorem \ref{thm4.3} that the last column has to consist of zeros
only. Moreover, since $\tilde\nabla^{\Cal T}$ is metric for $L(\tau)$
all the values of its curvature are skew symmetric with respect to
$L(\tau)$. Knowing that there are some zero blocks already, the
claimed form of the curvature is established by a simple direct
computation using formula \eqref{tract-met-split}.

For the second part of the proof, recall that for $n\geq 3$ it follows
from the general theory (see \cite{tractors} and \cite{confamb}) that
the canonical tractor connection on a conformal standard tractor
bundle is characterized by the fact that it is metric and its
curvature is normal. As described in the proof of Theorem
\ref{thm4.1a}, normality first requires that the curvature preserves
the canonical line subbundle $\Cal T^1$. If this is satisfied, one
obtains a tensor field describing the induced action of the curvature
on $(\Cal T^1)^\perp/\Cal T^1$. (For the tractor connection
$\tilde\nabla$ this is described by the component $W_{ij}{}^k{}_\ell$
from above.) The second part of the normality condition is that the
Ricci--type contraction of this component vanishes identically.

Now by Theorem \ref{thm4.3}, the tractor connection
$\tilde\nabla^{\Cal T}$ is metric for $L(\tau)$. From the definition
of $\tilde\Ps$, one easily verifies that for any vector field
$\xi$, the endomorphism $\tilde\Ps(\xi)$ is skew symmetric for
$L(\tau)$ and this immediately implies that $\nabla^0$ is a tractor
connection which is metric for $L(\tau)$. Hence to complete the proof,
it suffices to prove that the Ricci--type contraction of the tensor
describing the induced action of the curvature of $\nabla^0$ on
$(\Cal T^1)^\perp/\Cal T^1$ vanishes identically. 

Now as in the proof of Theorem \ref{thm4.3}, the curvature $R^0$ of
$\nabla^0$ is related to the curvature $\tilde R$ of
$\tilde\nabla^{\Cal T}$ by 
\begin{equation}\label{R01}
R^0(\xi,\eta)=\tilde R(\xi,\eta)+\tilde\nabla^{\End(\Cal
  T)}_\xi\tilde\Ps(\eta)-\tilde\nabla^{\End(\Cal
  T)}_\eta\tilde\Ps(\xi)-\tilde\Ps([\xi,\eta])+
[\tilde\Ps(\xi),\tilde\Ps(\eta)], 
\end{equation}
where $\tilde\nabla^{\End(\Cal T)}$ is the connection on $\End(\Cal
T)$ induced by $\tilde\nabla^{\Cal T}$ and the last bracket denotes
the commutator of endomorphisms. Now the fact that $\tilde\Ps$ is
concentrated in the block--lower--triangular part of the matrix says
that inserting any vector field into $\tilde\Ps$ one obtains a map
which vanishes on $\Cal T^1$, and maps $(\Cal T^1)^\perp$ to $\Cal
T^1$ (and also maps $\Cal T$ to $(\Cal T^1)^\perp$). This shows that
the last two terms in the right hand side of \eqref{R01} do not
contribute to the induced map on $(\Cal T^1)^\perp/\Cal T^1$ (and it
also implies that the last one always vanishes).

Next, it is a standard result on induced connections that the
definition of $\tilde\nabla^{\Cal T}$ in terms of $\nabla^{\Cal T}$
and the $\End(\Cal T)$--valued one--form $\Ps$ in Theorem \ref{thm4.3}
implies that
\begin{equation}\label{End-derivs}
\tilde\nabla^{\End(\Cal T)}_\xi\tilde\Ps(\eta)=\nabla^{\End(\Cal
  T)}_\xi\tilde\Ps(\eta)+[\Ps(\xi),\tilde\Ps(\eta)]. 
\end{equation}
Now we have already seen above, that $\tilde\Ps(\eta)$ maps $(\Cal
T^1)^\perp$ to $\Cal T^1$ and from the definition of $\Ps$ it follows
that $\Ps(\xi)$ vanishes on $\Cal T^1$. Thus, the composition
$\Ps(\xi)\o\tilde\Ps(\eta)$ does not contribute to the induced map on
$(\Cal T^1)^\perp/\Cal T^1$. To consider the composition
$\tilde\Ps(\eta)\o\Ps(\xi)$, it suffices to consider the image of
$\Ps(\xi)$ up to elements of $\Cal T^1$, which is described by the
tensor $A_a{}^b{}_c$ from Theorem \ref{thm4.3}. As we have noted in
the proof of that theorem $A_a{}^b{}_c\rho_b$ vanishes along the
boundary. This shows that $\Ps(\xi)$ has values in $(\Cal T^1)^\perp$,
so $\tilde\Ps(\eta)\o\Ps(\xi)$ has values in $\Cal T^1$ and does not
contribute to the action on $(\Cal T^1)^\perp/\Cal T^1$ either.

Collecting the information, we see that the difference
$R^0(\xi,\eta)-\tilde R(\xi,\eta)$ is given by 
$$
\nabla^{\End(\Cal T)}_\xi\tilde\Ps(\eta)-\nabla^{\End(\Cal
  T)}_\eta\tilde\Ps(\xi). 
$$
The first summand in this expression maps $s\in\Ga(\Cal T)$ to
$$
\nabla^{\Cal T}_\xi(\tilde\Ps(\eta)(s))-\tilde\Ps(\eta)(\nabla^{\Cal
  T}_\xi s).
$$ Now we can directly compute the induced action of this on $(\Cal
T^1)^\perp/\Cal T^1$ by applying this to an element of the form
$s=(0,\nu^\ell,0)$ and computing the middle slot of the result. We do
this in abstract index notation with the index $i$ corresponding to
$\xi$ and $j$ corresponding to $\eta$. For the first term, applying
$\tilde\Ps$ the result is $\ph_{j\ell}\nu^{\ell}$ in the bottom slot,
and zero in the two other slots, so differentiating by $\nabla_i^{\Cal
  T}$ according to \eqref{std-conn}, this produces
$\delta^k_i\ph_{j\ell}\nu^{\ell}$ in the middle slot. The middle slot
in the second term is given (including the sign) by multiplying
$\frac12\hat\tau^{-1}\ph_{jr}\ga^{kr}$ by the top slot of the
derivative in the bracket. By \eqref{split-ids}, the latter is given
by
$$
\hat\tau\rho_b\nabla_i\nu^b=-\hat\tau\nu^b\nabla_i\rho_b=
-\hat\tau 2\ga_{i\ell}\nu^\ell. 
$$
Collecting our results, we see that the tensor describing the action
of the curvature of $\nabla^0$ on $(\Cal T^1)^\perp/\Cal T^1$ is given
by
$$
W_{ij}{}^k{}_\ell+\delta^k_i\ph_{j\ell}-\delta^k_j\ph_{i\ell}-
\ph_{jr}\ga^{kr}\ga_{i\ell}+\ph_{ir}\ga^{kr}\ga_{j\ell}. 
$$
Forming the Ricci--type contraction, we get 
$$
W_{kj}{}^k{}_\ell+(n-2)\ph_{jl}+\ph_{kr}\ga^{kr}\ga_{j\ell},  
$$ 
and inserting the definition of $\ph_{j\ell}$ one immediately
verifies that this vanishes.
\end{proof}

\end{document}